\documentclass[11pt,a4paper]{article}

\usepackage[T1]{fontenc}
\usepackage{lmodern}
\usepackage[utf8]{inputenc}

\usepackage[left=2.45cm, top=2.45cm,bottom=2.45cm,right=2.45cm]{geometry}
\usepackage{amsmath}
\usepackage{amssymb,amsthm,graphicx,xcolor,tikz,pgfplots,bm,mathrsfs}

\usepackage[british]{babel}

\usepackage[
	bookmarks=true,
	bookmarksnumbered=true,
	bookmarksopen=true,
	unicode=true,
	pdftoolbar=true,
	pdfmenubar=true,
	pdffitwindow=false,
	pdfstartview={FitH},
	pdftitle={},
	pdfauthor={},
	pdfsubject={},
	pdfcreator={},
	pdfproducer={},
	pdfkeywords={},
	pdfnewwindow=true,
	colorlinks=true,
	linkcolor=black,
	citecolor=black,
	filecolor=black,
	urlcolor=black]{hyperref}

\makeatletter
	\def\@cite#1#2{[\textbf{#1}\if@tempswa , #2\fi]}	
	\def\@biblabel#1{[#1]}								
\makeatother

\newtheorem {theorem}{Theorem}

\newtheorem {lemma}[theorem]{Lemma}

\theoremstyle{definition}

\newtheorem {remark}[theorem]{Remark}

\newcommand{\EE}{\mathbb{E}}

\newcommand{\NN}{\mathbb{N}}
\newcommand{\PP}{\mathbb{P}}
\newcommand{\RR}{\mathbb{R}}

\newcommand{\ZZ}{\mathbb{Z}}

\def\sfN{{\sf N}}

\newcommand{\cA}{\mathcal{A}}
\newcommand{\cD}{\mathcal{D}}
\newcommand{\cB}{\mathcal{B}}
\newcommand{\cF}{\mathcal{F}}
\newcommand{\cE}{\mathcal{E}}
\newcommand{\cH}{\mathcal{H}}
\newcommand{\cK}{\mathcal{K}}
\newcommand{\cL}{\mathcal{L}}
\newcommand{\cM}{\mathcal{M}}
\newcommand{\cT}{\mathcal{T}}

\newcommand{\cG}{\mathcal{G}}

\DeclareMathOperator{\Vol}{Vol}

\DeclareMathOperator{\Var}{Var}

\DeclareMathOperator{\conv}{conv}

\newcommand{\ext}{\operatorname{ext}}
\newcommand{\aff}{\operatorname{aff}}

\newcommand{\cl}{\operatorname{cl}}
\newcommand{\pow}{\operatorname{pow}}
\newcommand{\inter}{\operatorname{int}}

\newcommand{\dint}{\mathrm{d}}
\newcommand{\skel}{\operatorname{
skel}}

\usepackage{accents}
\DeclareMathSymbol{\widetildesym}{\mathord}{largesymbols}{"65}

\begin{document}

\title{\bfseries The $\beta$-Delaunay tessellation IV:\\ Mixing properties and central limit theorems}

\author{Anna Gusakova\footnotemark[1],\; Zakhar Kabluchko\footnotemark[2],\; and Christoph Th\"ale\footnotemark[3]}

\date{}
\renewcommand{\thefootnote}{\fnsymbol{footnote}}
\footnotetext[1]{Ruhr University Bochum, Germany. Email: anna.gusakova@rub.de}

\footnotetext[2]{M\"unster University, Germany. Email: zakhar.kabluchko@uni-muenster.de}

\footnotetext[3]{Ruhr University Bochum, Germany. Email: christoph.thaele@rub.de}

\maketitle

\begin{abstract}
	\noindent Various mixing properties of $\beta$-, $\beta'$- and Gaussian Delaunay tessellations in $\mathbb{R}^{d-1}$ are studied. It is shown that these tessellation models are absolutely regular, or $\beta$-mixing. In the $\beta$- and the Gaussian case exponential bounds for the absolute regularity coefficients are found. In the $\beta'$-case these coefficients show a polynomial decay only. In the background are new and strong concentration bounds on the radius of stabilization of the underlying construction. Using a general device for absolutely regular stationary random tessellations, central limit theorems for a number of geometric parameters of $\beta$- and Gaussian Delaunay tessellations are established. This includes the number of $k$-dimensional faces and the $k$-volume of the $k$-skeleton for $k\in\{0,1,\ldots,d-1\}$.\\
	
		\noindent {\bf Keywords:} {Absolute regularity, beta-Delaunay tessellation, beta'-Delaunay tessellation, central limit theorem, Gaussian-Delaunay tessellation, mixing properties, radius of stabilization, stochastic geometry, tail triviality}\\
	{\bf MSC:} 52A22, 52B11, 53C65, 60D05, 60F05.
\end{abstract}

\section{Introduction}

In part I of this paper \cite{GKT1} we introduced two new classes of stationary and isotropic random simplicial tessellations in $\RR^{d-1}$, the so-called $\beta$- and $\beta'$-Delaunay tessellations, which generalize the construction of the classical and well-known Poisson-Delaunay tessellation \cite{SW}. The new models can informally be defined as follows. For simplicity, we focus on the case of the $\beta$-Delaunay tessellation in $\RR^{d-1}$, which similarly to the classical Poisson-Delaunay tessellation is also based on a Poisson point process $\eta_{\beta}$, but this time in the product space $\RR^{d-1}\times [0,\infty)$. Its intensity measure is a constant multiple of $h^\beta\,\dint v\dint h$, where $v\in\RR^{d-1}$ stands for the spatial coordinate and $h>0$ for the height coordinate of a point $x=(v,h)\in\RR^{d-1}\times [0,\infty)$, and $\beta>-1$ is a fixed parameter. In a next step, we construct the paraboloid hull of $\eta_{\beta}$. This is a special germ-grain process with paraboloid grains whose systematic study has been initiated in the work of Schreiber and Yukich \cite{SchreiberYukich}, and further been developed by Calka, Schreiber and Yukich \cite{CalkaSchreiberYukich} and Calka and Yukich \cite{CalkaYukichGaussian} in the context of random convex hulls. We use the paraboloid hull process to construct a random tessellation $\cD_\beta$ in $\RR^{d-1}$ with only simplicial cells as follows. Given $d$ points $x_1=(v_1,h_1),\ldots,x_d=(v_d,h_d)$ of $\eta_\beta$ with almost surely affinely independent spatial coordinates $v_1,\ldots,v_d$, there is a unique translate of the standard downward paraboloid
$$
\Pi^\downarrow := \left\{(v,h)\in\RR^{d-1}\times\RR\colon h\leq -\|v\|^2\right\}
$$
containing $x_1,\ldots,x_d$ on its boundary. We declare $\conv(v_1,\ldots,v_d)$ to be a simplex of the $\beta$-Delaunay tessellation in $\RR^{d-1}$ if and only if the interior of the downward paraboloid determined by $x_1,\ldots,x_d$ does not contain any point of $\eta_{\beta}$. The collection of all $\beta$-Delaunay simplices is called the $\beta$-Delaunay tessellation of $\RR^{d-1}$. This is a stationary and isotropic random tessellation of $\RR^{d-1}$ as we have shown in \cite[Section 3]{GKT1}. In parallel to the construction of $\beta$-Delaunay tessellations we also introduced in \cite{GKT1} the class of $\beta'$-Delaunay tessellations $\cD_{\beta}'$, which is based on a Poisson point process $\eta_\beta'$ on $\RR^{d-1}\times (-\infty,0)$ with intensity measure being a constant multiple of $(-h)^{-\beta}\,\dint v\dint h$ and where $\beta>(d+1)/2$ is a fixed parameter. Realizations of the $\beta$- and the $\beta'$-Delaunay tessellation are shown in Figure \ref{fig:tess}. Additionally, in part II \cite{GKT2} we proved that after suitable rescaling, as $\beta\to\infty$, the $\beta$- and the $\beta'$-Delaunay tessellation converge to a common stationary and isotropic random simplicial limiting tessellation in $\RR^{d-1}$, which we called the Gaussian-Delaunay tessellation. This terminology was motivated by the fact that the typical cell of the limiting tessellation could be identified in distribution with a volume-weighted Gaussian random simplex.

	\begin{figure}[t]
		\begin{center}
		\includegraphics*[width=0.45\columnwidth]{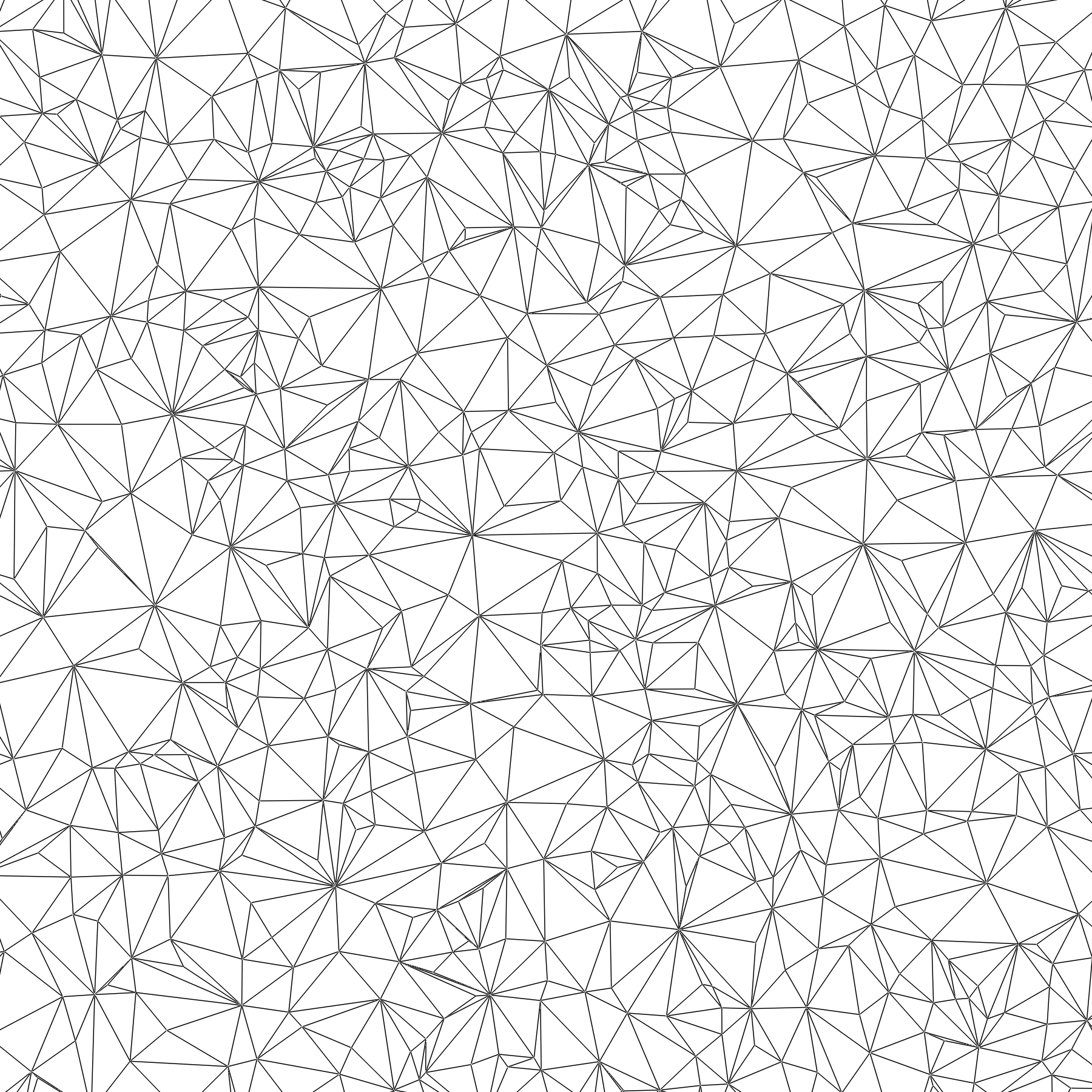}\qquad
		\includegraphics*[width=0.45\columnwidth]{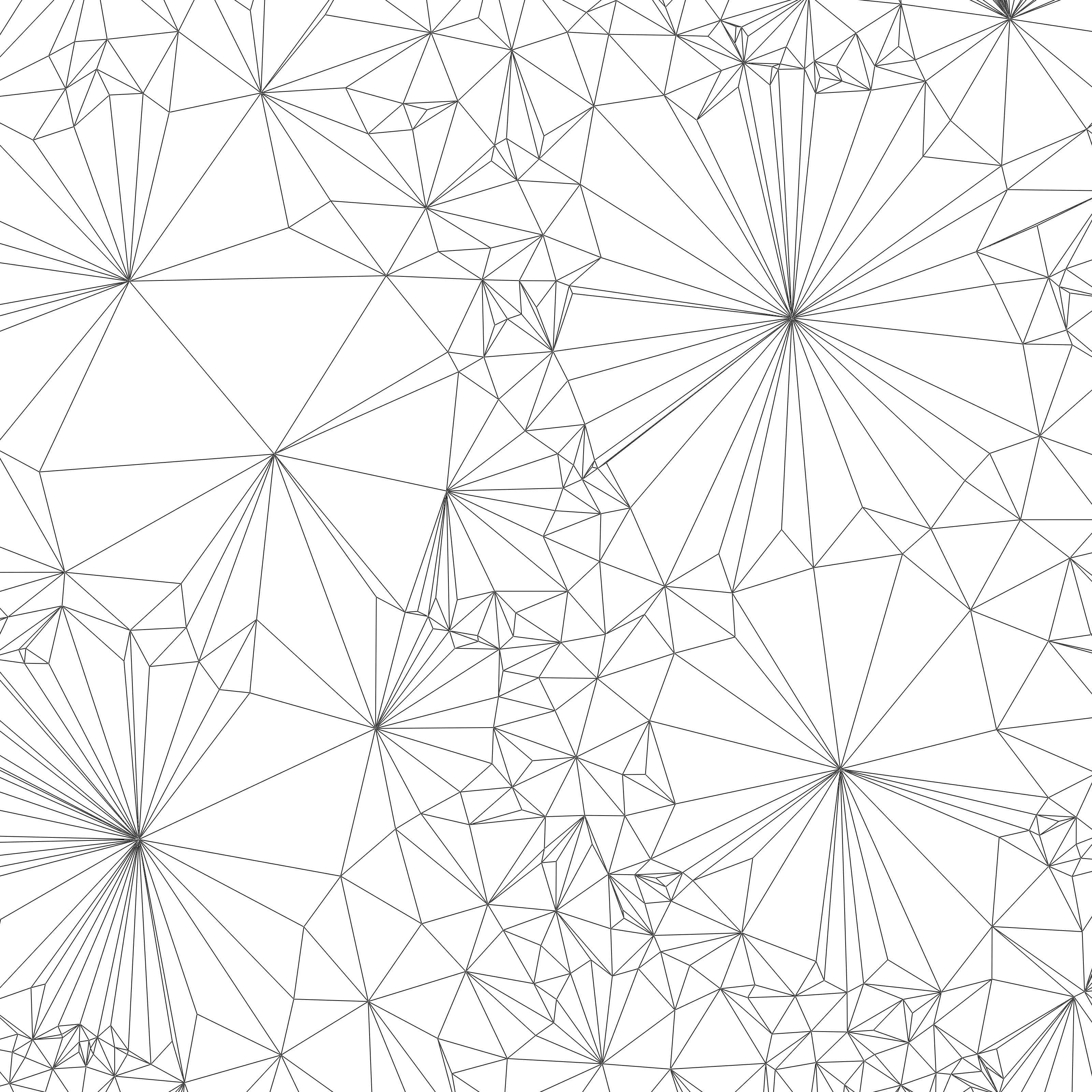}
		\label{fig:tess}
		\caption{Realizations of a $\beta$-Delaunay tessellation for $\beta=15$ (left) and a $\beta'$-Delaunay tessellation for $\beta=2.5$ in $\RR^2$.}
	\end{center}
	\end{figure}

We would like to emphasize that the $\beta$, $\beta'$- and Gaussian-Delaunay tessellations are of interest in the study of the asymptotic geometry of random convex hulls. For example, the $\beta$-Delaunay tessellation reflects the local asymptotic structure of so-called random $\beta$-polytopes in the $d$-dimensional unit ball close to its boundary. In fact, after suitable rescaling, the surface of the ball locally converges to its tangent space and the radial projection of the boundary of the random $\beta$-polytope converges to the  $\beta$-Delaunay tessellation in this tangent space. {A similar interpretation also holds for Gaussian random polytopes and the Gaussian-Delaunay tessellation. Finally, $\beta'$-Delaunay tessellations are conjectured to arise from the so-called beta$^*$ polytopes.}

The purpose of the present paper is the study of several mixing properties for all the random tessellation models mentioned so far, that is, the $\beta$-, the $\beta'$- and the Gaussian-Delaunay tessellation. More precisely, we study the following hierarchy of concepts:
\begin{itemize}
\item ergodicity,
\item mixing of any order, also known as ergodic mixing,
\item tail triviality, also known as Kolmogorov's K-property,
\item $\alpha$-mixing, also known as strong mixing,
\item absolute regularity, also known as $\beta$-mixing (a terminology we do not follow to avoid confusion with the model parameter $\beta$).
\end{itemize}
Our main result is an explicit exponential bound on the absolute regularity coefficient of a $\beta$- and Gaussian-Delaunay tessellation as well as a polynomial bound in case of the $\beta'$-Delaunay tessellation. The proof relies on strong estimates for the so-called radius of stabilization of our construction, which considerably refines previous results from part II of this paper \cite{GKT2} and are of independent interest. While for the purposes in part II much less explicit bounds were sufficient, in order to handle the absolute regularity coefficient in the way we need, much finer control on the radius of stabilization is required. In a sense, this is the technical heart of our paper.

As an application of our bounds on the absolute regularity coefficients we develop a central limit theory for combinatorial and also for geometric parameters associated with $\beta$- and Gaussian-Delaunay tessellations, using a rather general central limit theorem of Heinrich \cite{Heinrich94,HeinrichBook} for absolutely regular stationary random random fields. More precisely, we prove that the number of $k$-dimensional faces of these tessellation models that can be observed in a sequence of cubes of increasing edge lengths is asymptotically Gaussian. A similar result is also shown for the natural volume measure on the so-called $k$-skeleton, which is the union of all $k$-dimensional faces of all tessellation cells.
This complements the limit theorems studied in part III \cite{GKT3} of this series of papers, which focus on a high-dimensional setting.

\bigskip

The remaining parts of this text are structured as follows. In Section \ref{sec:Background} we collect some background material in order to keep the paper reasonably self-contained, whereas in Section \ref{sec:MixingConcepts} we formally recall the mixing concepts for stationary random closed sets mentioned above. Fine estimates for the radius of stabilization are developed in Section \ref{sec:Stabilization}, which, as anticipated above, is the technical heart of our paper. The results we develop there are used in Section \ref{sec:AbsoluteRegularity} to prove that the $\beta$-, the $\beta'$- and the Gaussian-Delaunay tessellations are absolutely regular and we also provide sharp bounds for their absolute regularity coefficients. In the final Section \ref{sec:CLTs} we use these bounds to deduce the announced central limit theorems for the number of $k$-dimensional faces and the empirical $k$-volume density.

\section{Background material}\label{sec:Background}

\subsection{Notation}

Denote by $B_a(w)$ the $(d-1)$-dimensional closed ball of radius $a$ centred at $w$ and by $B_a$ the $(d-1)$-dimensional radius-$a$-ball centred at zero. Let
$$
\kappa_{d-1}={\pi^{d-1\over 2}\over \Gamma (1+{d-1\over 2})}
$$
be the volume of $(d-1)$-dimensional unit ball. Given a set $K\subset \RR^{d-1}$ we denote by $\cl K$ the closure of $K$, by $\partial K$ its topological boundary, and by $\aff(K)$ and $\conv(K)$ its affine or convex hull, respectively.

For two functions $f,g:\RR^k\to\RR$ We will frequently use the notation $f(x_1,\ldots,x_k)\ll g(x_1,\ldots,x_k)$, which means that there exists a positive constant $c$, independent of $x_1,\ldots,x_k$, such that $f(x_1,\ldots,x_k)\leq c g(x_1,\ldots,x_k)$.

\subsection{Description of the tessellation models}

In this section we briefly recall the construction of the $\beta$-Delaunay, the $\beta'$-Delaunay and the Gaussian-Delaunay tessellation we introduced in part I \cite{GKT1} and II \cite{GKT2} of this series of papers. We also refer the reader to these texts for further background material, details and proofs.

We start by introducing the three Poisson point processes on which the construction of our tessellations is based. For $\beta>-1$ let $\eta_\beta$ be a Poisson point process on the product space $\RR^{d-1}\times[0,\infty)$ whose intensity measure has density
\begin{equation}\label{eq:intensityBeta}
(v,h)\mapsto c_{d,\beta}h^\beta,\qquad c_{d,\beta}:={\Gamma({d\over 2}+\beta+1)\over\pi^{d\over 2}\Gamma(\beta+1)},
\end{equation}
with respect to the Lebesgue measure on $\RR^{d-1}\times[0,\infty)$. Next, for $\beta>(d+1)/2$ let $\eta_\beta'$ be a Poisson point process on $\RR^{d-1}\times(-\infty,0]$ whose intensity measure has density
$$
(v,h)\mapsto c_{d,\beta}'(-h)^{(-\beta)},\qquad c'_{d,\beta}:={\Gamma(\beta)\over\pi^{d\over 2}\Gamma(\beta-{d\over 2})},
$$
with respect to the Lebesgue measure on $\RR^{d-1}\times(-\infty,0]$. Finally, we denote by $\zeta$ a Poisson point process on $\RR^{d-1}\times\RR$ whose intensity measure has density
\begin{equation}\label{eq:intensityGaussian}
(v,h)\mapsto {1\over(2\pi)^{d/2}}e^{h/2}
\end{equation}
with respect to the Lebesgue measure on that space. Whenever possible, we will treat these three cases in parallel and write $\xi$ in what follows for one of the Poisson processes $\eta_\beta$, $\eta_\beta'$ or $\zeta$.

We shall now briefly describe the two different ways to construct the $\beta^{(')}$- and Gaussian Delaunay tessellation. For the first approach, let $\Pi_{\pm}$ be the standard upward ($+$) or downward ($-$) paraboloid $\Pi_{\pm}:=\{(v,h)\in\RR^{d-1}\times\RR:h=\pm\|v\|^2\}$ and $\Pi_{\pm,x}$ be the shifted standard paraboloid with apex at $x\in\RR^d$. For a set $A\subset\RR^d$ we let $A^{\uparrow}:=\{(v,h')\in\RR^{d-1}\times\RR\colon (v,h) \in A \text{ for some } h\leq h'\}$ and $A^{\downarrow}:=\{(v,h')\in\RR^{d-1}\times\RR\colon (v,h) \in A \text{ for some } h\ge h'\}$. The paraboloid growth process $\Psi(\xi)$ based on $\xi$ is now defined as
$$
\Psi(\xi):=\bigcup_{x\in\xi}\Pi_{+,x}^\uparrow.
$$
In other words, $\Psi(\xi)$ is the Boolean model with paraboloid grains based on $\xi$. A point $x\in\xi$ is called extreme in $\Psi(\xi)$ if and only if its associated paraboloid is not fully covered by the paraboloids associated with other points of $\xi$. The collection of such extreme points is denoted by $\ext(\Psi(\xi))$.

Next, we introduce the paraboloid hull process, which can be regarded as dual to the paraboloid growth process. For any collection $x_1:=(v_1,h_1),\ldots,x_k:=(v_k,h_k)$ of $1\leq k\leq d$ points in $\RR^{d-1}\times\RR$ with affinely independent spatial coordinates $v_1,\ldots,v_k\in  \RR^{d-1}$, we define $\Pi(x_1,\ldots,x_k)$ to be the intersection of $\aff(v_1,\ldots,v_k)\times\RR$ with any translate of $\Pi_{-}$ containing all the points $x_1,\ldots,x_k$ (although such translates of $\Pi_{-}$ are not unique for $k<d$, their intersections with $\aff(v_1,\ldots,v_k)\times\RR$ all coincide so that $\Pi(x_1,\ldots,x_k)$ is well defined). Further, we define the parabolic face $\Pi[x_1,\ldots,x_k]$ as
\[
\Pi[x_1,\ldots,x_k]:=\Pi(x_1,\ldots,x_k)\cap \left(\conv(v_1,\ldots,v_k)\times\RR\right).
\]
The paraboloid hull $\Phi(\xi)$ of $\xi$ is given by
$$
	\Phi(\xi)=\bigcup\limits_{(x_1,\ldots,x_d)\in\xi_{\neq}^d}\left(\Pi[x_1,\ldots, x_d]\right)^{\uparrow},
$$
where $\xi_{\neq}^d$ is the collection of all $d$-tuples of distinct points of $\xi$. In this context, the points $x_1,\ldots,x_d$ are called vertices of the paraboloid face $\Pi[x_1,\ldots, x_d]$ and we let ${\rm vert}(\Phi(\xi))$ be the collection of vertices of all such paraboloid faces. The link between $\Psi(\xi)$ and $\Phi(\xi)$ arises from the fact that $\ext(\Psi(\xi))={\rm vert}(\Phi(\xi))$.

Based on the paraboloid hull processes $\Phi(\xi)$ we are now in the position to define a random tessellation of $\RR^{d-1}$ as follows. For any paraboloid face of $\Phi(\xi)$ with vertices $x_1,\ldots, x_m\in\xi$ we consider a polytope $\conv(v_1,\ldots,v_m)$ and let $\cM_{\Phi}(\xi)$ be a collection of all such polytopes. With probability one $\cM_{\Phi}(\xi)$ is a stationary random tessellation of $\RR^{d-1}$ only having simplicial cells. In this case the construction above can be alternatively described as follows: for any collection $x_1=(v_1,h_1),\ldots,x_d=(v_d,h_d)$ of pairwise distinct points from $\xi$ we say that the simplex $\conv(v_1,\ldots,v_d)$ belongs to $\cM_{\Phi}(\xi)$ if and only if $\inter\Pi(x_1,\ldots,x_d)\cap\eta_{\beta}=\emptyset$.

Another alternative description of $\cM_{\Phi}(\xi)$ can be made in terms of the dual of a Laguerre tessellation, as already mentioned in the introduction. To formally introduce this approach, for a point $w\in\RR^{d-1}$ and $(v,h)\in\xi$ define the power function $\pow(w,(v,h)):=\|v-w\|^2+h$
and the Laguerre cell
$$
C((v,h),\xi):=\{w\in\RR^{d-1}\colon \pow(w,(v,h))\leq \pow(w,(v',h'))\text{ for all }(v',h')\in \xi\}.
$$
We emphasize that it is not necessarily the case that a Laguerre cell is non-empty or that it contains interior points. The collection of all Laguerre cells of $\xi$ with non-vanishing interior is called the Laguerre tessellation of $\xi$ and we write
\[
\cL(\xi):=\{C((v,h),\xi)\colon (v,h)\in \xi,\; {\rm int}(C((v,h),\xi))\neq\varnothing\}.
\]
Let $\cL^*(\xi)$ be the dual tessellation of $\cL(\xi)$. This tessellation arises from $\cL(\xi)$ by including for distinct points $x_1=(v_1,h_1),\ldots,x_d=(v_d,h_d)$ of $\xi$ the simplex $\conv(v_1,\ldots,v_d)$ in $\cL^*(\xi)$ if and only if the Laguerre cells corresponding to $(v_1, h_1),\ldots,(v_d, h_d)$ all have non-empty interior and share a common point. In our case $\cL^*(\xi)$ is almost surely a stationary random simplicial tessellation and coincides with $\cM_{\Phi}(\xi)$ as introduced earlier.

Finally, we introduce the notation $\skel(\cL^*(\xi))=\skel(\cM_{\Phi}(\xi))$ for the random set in $\RR^{d-1}$ arising from the union of all cell boundaries of cells from $\cL^*(\xi)$ or $\cM_{\Phi}(\xi)$.

\section{Mixing concepts for stationary random closed sets}\label{sec:MixingConcepts}

In this section we recall the definition of various concepts of mixing properties for stationary random closed sets in $\RR^{d-1}$. To this end we denote by $\cF^{d-1}:=\cF(\RR^{d-1})$ the set of closed subsets of $\RR^{d-1}$ and by $\mathfrak{F}^{d-1}:=\mathfrak{F}(\RR^{d-1})$ the Borel $\sigma$-field on $\cF^{d-1}$ generated by the usual Fell topology, see \cite[Chapter 2]{SW}. For $\cA\in\mathfrak{F}^{d-1}$ and $v\in\RR^{d-1}$ we define $T_vA:=\{a+v:a\in A\}$, $A\in\cA$, and put $T_v\cA:=\{T_vA:A\in\cA\}$. From now on $Z\subset\RR^{d-1}$ will denote a stationary random closed set. We call $\cA\in\mathfrak{F}^{d-1}$ invariant if $\PP(\{Z\in\cA\}\Delta\{Z\in T_v\cA\})=0$ for all $v\in\RR^{d-1}$, where $\Delta$ denotes the symmetric difference. The stationary random closed set $Z$ is called \textbf{ergodic} if $\PP(Z\in \cA)\in\{0,1\}$ for any $\cA$ from the $\sigma$-algebra generated by the class of invariant sets. This is known to be equivalent to require that
$$
\lim_{n\to\infty}{1\over (2n)^{d-1}}\int_{[-n,n]^{d-1}}\PP(\{Z\in\cA_0\}\cap\{Z\in T_v\cA_1\})\,\dint v = \PP(Z\in\cA_0)\PP(Z\in \cA_1)
$$
for all $\cA_0,\cA_1\in\mathfrak{F}^{d-1}$.  The random closed set $Z$ is called (ergodic) \textbf{mixing} if
$$
\lim_{\|v\|\to\infty}\PP(\{Z\in\cA_0\}\cap\{Z\in T_v\cA_1\}) = \PP(Z\in\cA_0)\PP(Z\in\cA_1)
$$
for all $\cA_0,\cA_1\in\mathfrak{F}^{d-1}$, where $\|v\|\to\infty$ indicates that $v$ runs through an arbitrary sequence $(v_n)_{n\geq 1}\subset\RR^{d-1}$ satisfying $\|v_n\|\to\infty$, as $n\to\infty$. More generally, for integers $\ell\geq 1$ one calls $Z$ (ergodic) \textbf{mixing of order $\ell$} if
{
$$
\lim_{n\to\infty} \PP(\{Z\in\cA_0\}\cap \{Z\in T_{v_n^{(1)}}\cA_1\}\cap\ldots\cap \{Z\in T_{v_n^{(\ell)}}\cA_\ell\}) = \PP(Z\in\cA_0)\PP(Z\in\cA_1)\cdots\PP(Z\in\cA_\ell)
$$}
for all $\cA_0,\cA_1,\ldots,\cA_\ell\in\mathfrak{F}^{d-1}$, where the limit runs through arbitrary sequences $(v_n^{(i)})_{n\geq 1}$, $i\in\{1,\ldots,\ell\}$, satisfying $\|v_n^{(i)}\|\to\infty$ for $i\in\{1,\ldots,\ell\}$ and $\|v_n^{(i)}-v_n^{(j)}\|\to\infty$ for $i,j\in\{1,\ldots,\ell\}$ with $i\neq j$, as $n\to\infty$. Note that saying that $Z$ is mixing of order $1$ is the same as that $Z$ is just mixing.

Next, we define for $n\geq 1$ the following sub-$\sigma$-algebras of $\mathfrak{F}^{d-1}$:
\begin{align*}
\mathfrak{F}_{n}^{d-1} &:=\sigma\{\{F\in\cF^{d-1}:F\cap C=\varnothing\}:C\subset B_n\text{ compact}\},\\
\mathfrak{F}_{-n}^{d-1} &:=\sigma\{\{F\in\cF^{d-1}:F\cap C=\varnothing\}:C\subset\RR^{d-1}\setminus B_n\text{ compact}\},
\end{align*}
as well as the so-called \textbf{tail $\sigma$-algebra}
$$
\mathfrak{F}_{-\infty}^{d-1} := \bigcap_{n=1}^\infty\mathfrak{F}_{-n}.
$$
The stationary random closet set $Z$ is called \textbf{tail-trivial}, provided that $\PP(Z\in \cA)\in\{0,1\}$ for all $\cA\in\mathfrak{F}_{-\infty}^{d-1}$.

There are also two stronger types of mixing conditions we consider here, namely $\alpha$-mixing and absolute regularity (also known as $\beta$-mixing, but we do not follow this notion here to avoid confusion with our model parameter $\beta$). Given $0<a<b$ we define the two functions
$$
\mathscr{A}(a,b):=\sup\limits_{\cA\in\mathfrak{F}_{a}^{d-1}\atop \cA^{\prime}\in\mathfrak{F}_{-b}^{d-1}}|\PP(\{Z\in\cA\}\cap\{Z\in\cA^{\prime}\})-\PP(Z\in\cA)\PP(Z\in\cA^{\prime})|,
$$
and
$$
\mathscr{B}(a,b):={1\over 2}\sup\sum\limits_{i=1}^{I}\sum\limits_{j=1}^{J}|\PP(\{Z\in\cA_i\}\cap\{Z\in\cA^{\prime}_j\})-\PP(Z\in\cA_i)\PP(Z\in\cA^{\prime}_j)|,
$$
where the supremum in the definition of $\mathscr{B}(a,b)$ is taken over all pairs of finite partitions of $\cF^{d-1}$: $\bar{\cA}:=\{\cA_i\in\mathfrak{F}_{a}^{d-1}\colon i=1,\ldots,I\}$, $\bar{\cA^{\prime}}:=\{\cA_j^{\prime}\in\mathfrak{F}_{-b}^{d-1}\colon j=1,\ldots, J\}$. The stationary random closed set $Z$ is called \textbf {$\alpha$-mixing} if for any $a>0$, $\lim\limits_{b\to\infty}\mathscr{A}(a,b)=0$ and $Z$ is called \textbf {absolutely regular} if for any $a>0$, $\lim\limits_{b\to\infty}\mathscr{B}(a,b)=0$.

The next lemma describes the relations between these different concepts.

\begin{lemma}
Let $Z$ be a stationary random closet set in $\RR^{d-1}$.
\begin{itemize}
	\item[(i)] If $Z$ is mixing then $Z$ is ergodic.
	\item[(ii)] If $Z$ is mixing of some order $\ell\geq 1$ then $Z$ is mixing.
	\item[(iii)] If $Z$ is tail-trivial then $Z$ is mixing of any order $\ell\geq 1$.
	\item[(iv)] If $Z$ is $\alpha$-mixing then $Z$ is tail-trivial.
	\item[(v)] If $Z$ is absolutely regular then $Z$ is $\alpha$-mixing.
\end{itemize}
\end{lemma}
\begin{proof}
That the mixing property of $Z$ implies ergodicity is well known , see \cite[p.\ 408]{SW}. Also, from the definition it directly follows that if $Z$ is mixing of some order $\ell\geq 1$ then $Z$ is just mixing. Next, we need to prove that if $Z$ is tail-trivial then $Z$ is mixing of any order $\ell\geq 1$. For stationary random measures this is the content of \cite[Theorem 6.3.6]{KMK} and the proof carries over literally to stationary random closed sets (see also \cite[Exercise 12.3.7]{DVJ}). That $\alpha$-mixing implies tail-triviality can be found in \cite[Section 2.5, p.\ 116]{Bradley}. The final assertion is a consequence of the inequality $\mathscr{A}(a,b)\leq{1\over 2}\mathscr{B}(a,b)$ for which we refer to \cite[Equation (1.11)]{Bradley}.
\end{proof}

In particular, the last lemma shows that once absolute regularity of a stationary random closed set is established, all other (weaker) mixing notions follow automatically. This is the reason why from now on we focus on absolute regularity only. In order to establish the absolute regularity property we will rely on the following simple lemma we took from \cite[Lemma 1]{MN15}.

\begin{lemma}\label{lm:betamixing}
Let $0<a<b$ and $\varepsilon >0$. If for some $\cB\in\mathfrak{F}_{a}^{d-1}$ it holds that for any $\cA^{\prime}\in\mathfrak{F}^{d-1}_{-b}$ and for any $\cA\in\mathfrak{F}^{d-1}_a$ with $\cA\subseteq\cB$, $\PP(Z\in\cA)>0$ we have
$$
|\PP(Z\in\cA^{\prime}|Z\in\cA)-\PP(Z\in\cA^{\prime})|\leq\varepsilon,
$$
then $\mathscr{B}(a,b)\leq \epsilon\PP(Z\in\cB)+\PP(Z\not\in\cB)$.
\end{lemma}

\section{Fine estimates for the radius of stabilization}\label{sec:Stabilization}

In this section we establish a kind of strong localization property for $\beta$-, $\beta^{\prime}$- and Gaussian-Delaunay tessellations in the spirit of the geometric limit theory of stabilization for which we refer to the survey articles \cite{SchreiberSurvey,YukichSurvey}. This result is the technical heart of the present paper and is the crucial ingredient when we establish our bounds for the absolute regularity coefficients in Section \ref{sec:AbsoluteRegularity}. We remark that the following two lemmas were are already content of part II of this series of papers (see Lemma 4.4 and Lemma 4.5 in \cite{GKT2}), but we need them here in a much more precise and refined form. The notation follows that of part II of our paper.

\begin{lemma}\label{lm:4_3}
The following assertions hold.
	\begin{enumerate}
	\item
	\begin{enumerate}
	\item For any $\beta > -1$, $A>0$ and $T>0$ we have
	$$
	\PP\Big(\sup\limits_{\substack{(v,h)\in\partial \Psi(\eta_{\beta}), v\in B_A}} h > T\Big)\leq
	\begin{cases}
	\exp\big(-{c_{d,\beta}\kappa_{d-1}\over \beta+1}A^{d-1}(T-4A^2)^{\beta+1}\big) &:T>4A^2\\
	1 &: T\leq 4A^2;
	\end{cases}
	$$
	\item for any $\beta > (d+1)/2$, $A>0$ and $T>0$ we have
    $$
	\PP\Big(\sup\limits_{\substack{(v,h)\in\partial \Psi(\eta^{\prime}_{\beta}), v\in B_A}} h > T\Big)\leq
	\begin{cases}
	\exp\big(-{c_{d,\beta}^{\prime}\kappa_{d-1}\over \beta-1}A^{d-1}(4A^2-T)^{1-\beta}\big) &: T< 0\\
	0 &: T\ge 0;
	\end{cases}
	$$
	\item for any $A>0$ and $T\in\RR$ we have
	$$
	\PP\Big(\sup\limits_{\substack{(v,h)\in\partial \Psi(\zeta),v\in B_A}} h > T\Big)<\exp\Big(-2\cdot(2\pi)^{-d/2}\kappa_{d-1}A^{d-1}e^{T/2-2A^2}\Big).
	$$
	\end{enumerate}
	\item
	\begin{enumerate}
	\item For any $\beta >-1 $ and $A>0$ we have
	$$
	\PP\Big(\inf\limits_{\substack{(v,h)\in\partial \Psi(\eta_{\beta}),v\in B_A}} h < t\Big)<\begin{cases}
	1-\exp\Big(-{\Gamma({d\over 2}+\beta+1)\over \Gamma({d+1\over 2})}t^{\beta+1}(A+\sqrt{t})^{d-1}\Big) &: t>0\\
	0 &:t\leq 0;
	\end{cases}
	$$
	\item for any $\beta> (d+1)/2$ and $A>0$ we have
	$$
	\PP\Big(\inf\limits_{\substack{(v,h)\in\partial \Psi(\eta_{\beta}^{\prime}),v\in B_A}} h < t\Big)<\begin{cases}
	1-\exp\Big(-{\max(\Gamma(\beta-{d+1\over 2}), \Gamma(\beta-1))\over \Gamma(\beta-{d\over 2})}\\ \qquad\qquad\qquad\qquad\times|t|^{1-\beta}(A+\sqrt{|t|})^{d-1}\Big) &: t<0\\
	1 &:t\ge 0;
	\end{cases}
	$$
	\item for any $A>0$ and $t\in\RR$ we have
	$$
	\PP\Big(\inf\limits_{\substack{(v,h)\in\partial \Psi(\zeta),v\in B_A}} h < t\Big)<1-\exp(-2\pi^{-1/2}(A+1)^{d-1}e^{t/2}).
	$$
	\end{enumerate}
	\end{enumerate}
\end{lemma}

\begin{proof}[Proof of Lemma \ref{lm:4_3}]
The proof is basically the same as the proof of Lemma 4.3 in part II \cite{GKT2}, since up to certain point it works for an arbitrary Poisson point process $\xi$ in the product space $\RR^{d-1}\times\RR$. In particular from \cite[Equation (4.3)]{GKT2}, we obtain
$$
	\begin{aligned}
	\PP\Big(\sup\limits_{\substack{(v,h)\in\partial \Psi(\xi),v\in B_A)}} h > T\Big)
	&\leq\exp(-\EE[\xi\cap (B_A\times (-\infty, T-4A^2])]).
	\end{aligned}
$$	
	It remains to evaluate the last term for the three Poisson point processes $\xi=\eta_{\beta}$, $\xi=\eta^{\prime}_{\beta}$ and $\xi=\zeta$. We concentrate on the computation in the first two cases, since the last one was already carried out in part II \cite[Lemma 4.4(1c, 2c)]{GKT2}.
	
	From the definition of the intensity measure of the Poisson point processes $\eta_{\beta}$ for $T>4A^2$ we get
	\begin{align*}
	 \EE[\eta_{\beta}\cap (B_A\times (-\infty, T-4A^2])]&= c_{d,\beta}\int_{B_A}\int_{0}^{T-4A^2}h^{\beta}\,\dint h\dint v
	 = {c_{d,\beta}\kappa_{d-1}\over \beta+1}A^{d-1}(T-4A^2)^{\beta+1},
	\end{align*}
	otherwise $\EE[\eta_{\beta}\cap (B_A\times (-\infty, T-4A^2])]=0$. Analogously in case of $\eta_{\beta}^{\prime}$ we obtain for $T< 0$,
	\begin{align*}
	 \EE[\eta^{\prime}_{\beta}\cap (B_A\times (-\infty, T-4A^2])]&= c_{d,\beta}^{\prime}\int_{B_A}\int_{4A^2-T}^{\infty}h^{-\beta}\,\dint h\dint v
	 ={c_{d,\beta}^{\prime}\kappa_{d-1}\over \beta-1}A^{d-1}(4A^2-T)^{1-\beta}.
	\end{align*}

For the second part of the lemma we use the equality
$$
\PP\Big(\inf\limits_{\substack{(v,h)\in\partial \Psi(\xi),v\in B_A}} h < t\Big)=1-\exp(-\EE[\xi\cap K(A,t)]),
$$
where
\begin{align*}
K(A,t)&:=(B_A\times (-\infty,t])\cup \Big(\bigcup\limits_{w\in\partial B_A}\Pi_{-,(w,t)}^{\downarrow}\Big)\\
&=\{(v,h)\in\RR^{d-1}\times (-\infty,t]\colon v\in B_{A+\sqrt{t-h}}\},
\end{align*}
see \cite[Equation (4.6)]{GKT2}.
Consider $\xi=\eta_{\beta}$. For $t>0$ we have
\begin{align*}
	 \EE[\eta_{\beta}\cap K(A,t)]&= c_{d,\beta}\int_{0}^{t}\int_{\{v\in\RR^{d-1}:\|v\|\leq \sqrt{t-h}+A\}}h^{\beta}\,\dint v\dint \\
	 &=c_{d,\beta}\kappa_{d-1}\int_{0}^{t}(\sqrt{t-h}+A)^{d-1}h^{\beta}\,\dint h\\
	 &=c_{d,\beta}\kappa_{d-1}\sum\limits_{i=0}^{d-1}{d-1\choose i}A^{i}\int_{0}^{t}(t-h)^{(d-1-i)/2}h^{\beta}\,\dint h.
\end{align*}
Since for any integer $0\leq i\leq d-1$ it holds that
\begin{align*}
\int_{0}^{t}(t-h)^{(d-1-i)/2}h^{\beta}\,\dint h&=t^{{d+1-i\over 2}+\beta}\int_{0}^{1}(1-s)^{(d-1-i)/2}s^{\beta}\,\dint s=t^{{d+1-i\over 2}+\beta}{\Gamma({d+1-i\over 2})\Gamma(\beta+1)\over \Gamma({d+1-i\over 2}+\beta+1)},
\end{align*}
and the minimum of $\Gamma(x)$ for $x>0$ is around $0.88$ we conclude that
\begin{align}
	 \EE[\eta_{\beta}\cap K(A,t)]&={\Gamma({d\over 2}+\beta+1)\over \sqrt{\pi}\Gamma({d+1\over 2})}\sum\limits_{i=0}^{d-1}{d-1\choose i}A^{i}t^{{d+1-i\over 2}+\beta}{\Gamma({d+1-i\over 2})\over \Gamma({d+1-i\over 2}+\beta+1)} \label{eq:12.07.21_1}\\
	 &<{\Gamma({d\over 2}+\beta+1)\over \Gamma({d+1\over 2})}t^{\beta+1}(A+\sqrt{t})^{d-1}.\notag
	 \end{align}
	
	 In the same way we deal with the case $\xi=\eta^{\prime}_{\beta}$. For $t<0$ we have
\begin{align*}
	 \EE[\eta^{\prime}_{\beta}\cap K(A,t)] &=c_{d,\beta}^{\prime}\kappa_{d-1}\sum\limits_{i=0}^{d-1}{d-1\choose i}A^{i}\int_{|t|}^{\infty}(h-|t|)^{(d-1-i)/2}h^{-\beta}\,\dint h.
\end{align*}
Since for any integer $0\leq i\leq d-1$ it holds that
\begin{align*}
\int_{|t|}^{\infty}(h-|t|)^{(d-1-i)/2}h^{-\beta}\,\dint h&=|t|^{{d+1-i\over 2}-\beta}\int_{0}^{1}(1-s)^{(d-1-i)/2}s^{\beta-{d+1-i\over 2}-1}\,\dint s\\
&=|t|^{{d+1-i\over 2}-\beta}{\Gamma({d+1-i\over 2})\Gamma(\beta-{d+1-i\over 2})\over \Gamma(\beta)},
\end{align*}
we conclude that
\begin{align*}
	 \EE[\eta^{\prime}_{\beta}\cap K(A,t)]&={1\over \sqrt{\pi}}\sum\limits_{i=0}^{d-1}{d-1\choose i}A^{i}|t|^{{d+1-i\over 2}-\beta}{\Gamma({d+1-i\over 2})\over \Gamma({d+1\over 2})}{\Gamma(\beta-{d+1-i\over 2})\over \Gamma(\beta-{d\over 2})}\\
	 &\leq {\max(\Gamma(\beta-{d+1\over 2}), \Gamma(\beta-1))\over \Gamma(\beta-{d\over 2})}|t|^{1-\beta}(A+\sqrt{|t|})^{d-1}.
	 \end{align*}
	This completes the proof of the lemma.
\end{proof}

In what follows we consider the random closed sets
$$
\widetilde{\cD}(\eta_{\beta})=\widetilde{\cD}_{\beta}:=\skel(\cL^*(\eta_{\beta})),\qquad \widetilde{\cD}(\eta^{\prime}_{\beta})=\widetilde{\cD}^{\prime}_{\beta}:=\skel(\cL^*(\eta^{\prime}_{\beta})),\qquad\widetilde{\cD}(\zeta)=\widetilde{\cD}:=\skel(\cL^*(\zeta))
$$
in $\RR^{d-1}$, where we use the same notation as in part II of our paper. The next lemma is a substantial refinement of \cite[Lemma 4.5]{GKT2}. Such a  quantitative version is required in order to give sharp bounds on the absolute regularity coefficients in the next section. We recall that for $R\geq 1$ and $r>0$ we say that $\widetilde{\cD}_{\beta}\cap B_R$ \textit{is determined by particles $(v,h)\in\partial \Psi(\eta_{\beta})$ with $v\in B_{R+r}$}, provided that the $\beta$-Delaunay tessellation within $B_R$ is unaffected by changes of the point configuration $\eta_{\beta}$ outside of $B_{R+r}\times \RR$. The same terminology is also applied if $\eta_{\beta}$ is replaced by one of the point processes $\eta^{\prime}_\beta$ or $\zeta$.

\begin{lemma}\label{lm:4_4}
\begin{enumerate}
\item For any $R>0$, $r\ge 1/3$, $\beta>-1$, $d\ge 2$ and for any constant $c_0>0$ we have
\begin{align*}
\PP(\widetilde{\cD}_{\beta}\cap B_R &\text{ is not determined by particles } (v,h)\in\partial \Psi(\zeta_{\beta})\text{ with } v\in B_{R+r})\\
&< \begin{cases}
c_2re^{-c_1r^{d+1+2\beta}}\quad &:\quad r\ge c_0R,\\
\tilde c_2 R^{d-2}r^{-d+3}e^{-\tilde c_1r^{d+1+2\beta}}\quad &:\quad r< c_0R,
\end{cases}
\end{align*}
for some positive constants $c_1,\tilde c_1, c_2, \tilde c_2$, depending on $d$, $\beta$ and $c_0$ only.
\item For any $R>0$, $r>0$, $d\ge 2$ and $\beta >(d+1)/2$, and for any constant $c_0>0$ we have
\begin{align*}
\PP(\widetilde{\cD}^{\prime}_{\beta}\cap B_R &\text{ is not  determined by particles } (v,h)\in\partial \Psi(\zeta_{\beta}^{\prime})\text{ with } v\in B_{R+r})\\
&< \begin{cases}
c_3r^{-(2\beta-d-1)}\quad &:\quad r\ge c_0R,\\
\tilde c_3 R^{d-2}r^{-2\beta+3}\quad &:\quad r< c_0R,
\end{cases}
\end{align*}
for some positive constants $c_3,\tilde c_3$, depending on $d$, $\beta$ and $c_0$ only.
\item For any $R>0$, $r\ge 1/3$, and $d\ge 2$, and for any constant $c_0>0$ we have
\begin{align*}
\PP(\widetilde{\cD}\cap B_R &\text{ is not determined by particles } (v,h)\in\partial \Psi(\zeta)\text{ with }v\in B_{R+r})\\
&< \begin{cases}
c_5e^{-c_4r^2}\quad &:\quad r\ge c_0R,\\
\tilde c_5 R^{d-2}r^{-d+2}e^{-\tilde c_4r^2}\quad &:\quad r< c_0R,
\end{cases}
\end{align*}
for some positive constants $c_4,\tilde c_4, c_5, \tilde c_5$, depending on $d$ and $c_0$ only.
\end{enumerate}
\end{lemma}
\begin{proof}
In what follows, let us write $\xi$ for one of the Poisson point processes $\zeta$, $\eta_{\beta}$ or $\eta^{\prime}_{\beta}$.
 Throughout the proof $C$ will denote a positive constant, which only depends on $d$, $\beta$ and $c_0$.  Its exact value might be different from case to case. The beginning of the proof coincides with the proof of Lemma 4.5 in \cite{GKT2}. For convenience we repeat the arguments here.

\begin{figure}[t]
\begin{center}
	\includegraphics[width=0.25\columnwidth]{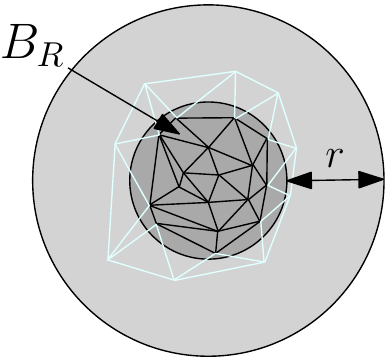}
\end{center}
\caption{Illustration of the event $\cE(\xi)$. Indicated is the random closed set $\widetilde\cD(\xi)$ inside the ball $B_R$, which under $\cE(\xi)$ is determined by particles of $\partial\Psi(\xi)$ with spatial coordinate in $B_{R+r}$.}
\label{fig:StabInside}
\end{figure}

For $R>0$ and $r>0$ we consider the event
$$
\cE(\xi):=\{\widetilde{\cD}(\xi)\cap B_R \text{ is not determined by particles } (v,h)\in\partial \Psi(\xi)\text{ with }v\in B_{R+r}\},
$$
see Figure \ref{fig:StabInside}. By the law of total probability we have that
\begin{equation}\label{eq_19.10_1}
\PP(\cE(\xi))\leq P_1(\xi)+P_2(\xi),
\end{equation}
where
\begin{align*}
P_1(\xi)&:= \PP\big(\cE(\xi) \, \big| \, \inf\limits_{\substack{(v,h)\in\partial \Psi(\xi),v\in \cl(B_{R+r}\setminus B_R)}} h > t\big),\\
P_2(\xi)&:=\PP\big(\inf\limits_{\substack{(v,h)\in\partial \Psi(\xi),v\in \cl(B_{R+r}\setminus B_R)}} h < t\big).
\end{align*}
These two terms are now dealt with separately, but we start with some general observations.

According to the construction, $\widetilde{\cD}(\xi)\cap B_R $ coincides almost surely with the skeleton of random tessellation $\cM_{\Phi}$. Then  $\widetilde{\cD}(\xi)\cap B_R $ is determined as soon as we know the location of all vertices of paraboloid facets of $\Phi(\xi)$ hitting the set $B_R\times \RR$. Then the event $\cE(\xi)$ occurs if and only if there is paraboloid facet of $\Phi(\xi)$ hitting the set $B_R\times \RR$ and having a vertex $(v,h)$ whose spatial coordinate satisfies $v\not \in B_{R+r}$. For $(v',h')\in\RR^{d-1}\times\RR$ let $\Pi_{-}(v',h')$ be a paraboloid such that $F(v',h'):=\Pi_{-}(v',h')\cap \Phi(\xi)$ is a paraboloid facet of $\Phi(\xi)$. Assume that $(B_R\times \RR)\cap F(v',h')\neq \varnothing$ and that there is a vertex $(v,h)\in F(v',h')$ with $v\not \in B_{R+r}$. Then $(B_{R+r}\times \RR)\cap F(v',h')\neq \varnothing$. Further, we note that the set $F(v',h')$ is connected, paraboloid convex, $F(v',h')\subset \partial \Phi(\xi)$ and  that
$$
\inf\limits_{\substack{(v,h)\in\partial \Phi(\xi),v\in \cl(B_{R+r}\setminus B_R)}} h\ge \inf\limits_{\substack{(v,h)\in\partial \Psi(\xi),v\in \cl(B_{R+r}\setminus B_R)}} h.
$$
This implies that conditionally on
$$
\inf\limits_{\substack{(v,h)\in\partial \Psi(\xi),v\in \cl(B_{R+r}\setminus B_R)}} h > t,
$$
we have
$$
F(v',h')\cap (\cl(B_{R+r}\setminus B_R)\times \RR)=F(v',h')\cap (\cl(B_{R+r}\setminus B_R)\times (t,\infty))\neq \varnothing,
$$
and, hence, there are two points $v_1\in \RR^{d-1}\setminus B_{R+r}$ and $v_2\in B_R$ such that for any $v_s=sv_1+(1-s)v_2$, $s\in[0,1]$ we have $(\{v_s\}\times [t,\infty))\cap F(v',h')\neq \varnothing$. Then for any $s\in[0,1]$ we have $v_s\in \Pi_{-}(v',h')\cap (\RR^{d-1}\times \{t\})$, which in turn means that $B_{\sqrt{h'-t}}(v')\cap B_R\neq \varnothing$ and $B_{\sqrt{h'-t}}(v')\cap (\RR^{d-1}\setminus B_{R+r})\neq \varnothing$, see Figure \ref{fig:proof_fig1}.

\begin{figure}[t]
	\centering
	\includegraphics[width=0.6\columnwidth]{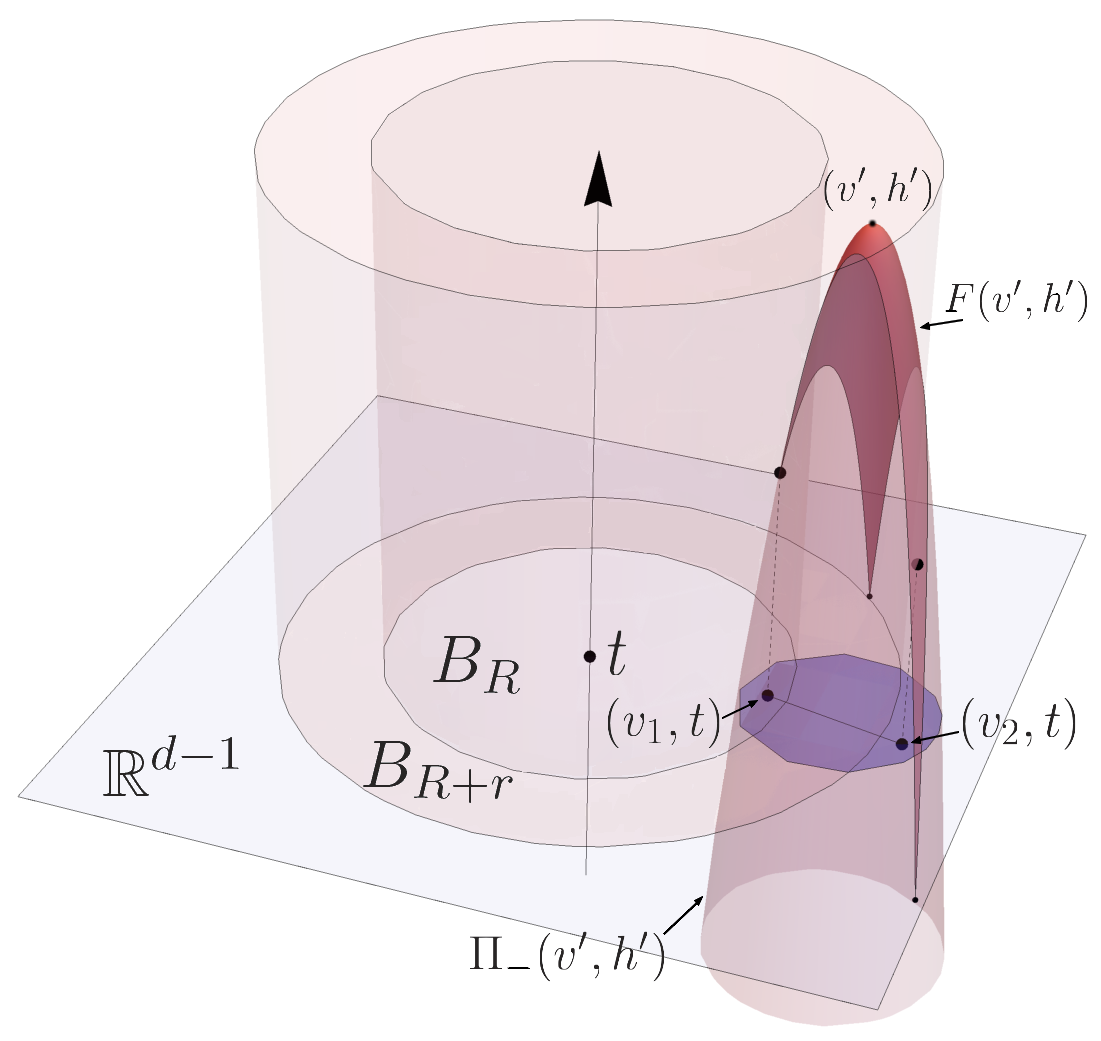}
	\caption{Illustration of the paraboloid facet $F(v',h')$ hitting the set $B_R\times\RR$ and having a vertex $(v,h)$ with $v\not\in B_{R+r}$.}
	\label{fig:proof_fig1}
\end{figure}

From the above argument it follows that
\begin{align*}
P_1(\xi)&\leq \PP\big(\exists\,\text{ a paraboloid facet } F(v',h') \text{ of } \Phi(\xi)\colon \\
&\hspace{1cm}B_{\sqrt{h'-t}}(v')\cap B_R\neq \varnothing, B_{\sqrt{h'-t}}(v')\cap (\RR^{d-1}\setminus B_{R+r})\neq \varnothing\big).
\end{align*}
Next, we observe that the condition $B_{\sqrt{h'-t}}(v')\cap B_R\neq \varnothing$ is equivalent to
$$
(v',h')\in \{(v,h)\in \RR^{d-1}\times [t,\infty)\colon v\in B_{R+\sqrt{h-t}}\}=:K_1(R,t),
$$
and the condition $B_{\sqrt{h'-t}}(v')\cap (\RR^{d-1}\setminus B_{R+r})\neq \varnothing$ is equivalent to
$$
(v',h')\in \{(v,h)\in \RR^{d-1}\times [t,\infty)\colon v\in (\RR^{d-1}\setminus B_{R+r-\sqrt{h-t}})\}=:K_2(R+r,t).
$$
Thus,
\begin{align}
P_1(\xi)&\leq \PP(\exists\,\text{ a paraboloid facet } F(v',h') \text{ of } \Phi(\xi)\colon (v',h')\in K_1(R,t)\cap K_2(R+r,t)).\label{eq_19.10_2}
\end{align}
Proceeding further as in the proof of Lemma 4.5 of the part II \cite{GKT2} of our paper consider a partition of $\RR^{d-1}$ into the boxes $Q_{x,y}$, $x\in a\ZZ^{d-1}$ with $a>0$ and $y\in \ZZ$, of the form
$$
Q_{x,y}:=(x\oplus [0,a]^{d-1})\times [t+r^2/4+by,t+r^2/4+b(y+1)],
$$
where $a>0$ is allowed to depend on $y$, $b>0$ and $\oplus$ denotes Minkowski addition. For fixed $y\in \ZZ$, $y\ge 0$, denote by $Q_{x_1,y},\ldots, Q_{x_{m(y)},y}$ the finite collection of $m(y)>0$ boxes having non-empty intersection with the set
\begin{align*}
K(t,R,r,y)&:=\{(v,h)\in K_1(R,t)\cap K_2(R+r,t)\colon  h \in [t+r^2/4+by,t+r^2/4+b(y+1)]\}\\
&=\{(v,h)\colon  h \in [t+r^2/4+by,t+r^2/4+b(y+1)], v\in (B_{R+\sqrt{h-t}}\setminus B_{R+r-\sqrt{h-t}})\},
\end{align*}
where we use the convention that $B_{M}=\varnothing$ whenever $M\leq 0$. Moreover, for $y <0$ the set $K(t,R,r,y)$ is empty. Further, for $y\geq 0$ and $x\in \RR^{d-1}$ we obtain
\begin{align*}
\PP\Big(\sup\limits_{\substack{(v,h)\in\partial \Psi(\xi),v\in (x\oplus [0,a]^{d-1})}} h > t+r^2/4+by\Big)&=\PP\Big(\sup\limits_{\substack{(v,h)\in\partial \Psi(\xi),v\in [0,a]^{d-1}}} h > t+r^2/4+by\Big)\\
&\leq \PP\Big(\sup\limits_{\substack{(v,h)\in\partial \Psi(\xi),v\in B_{\sqrt{d-1}a}}} h > t+r^2/4+by\Big)\\
&=:p_{a,b}(\xi,y).
\end{align*}
Thus, applying the union bound to \eqref{eq_19.10_2} and using the fact that apexes of the parabolic facets of the paraboloid hull process $\Phi(\xi)$ belong to the boundary of paraboloid growth process $\Psi(\xi)$ we see that
\begin{equation}\label{eq_19.10_3}
\begin{aligned}
P_1(\xi)&< \sum_{y=0}^{\infty}\sum_{j=0}^{m(y)}\PP(\exists\,\text{ a paraboloid facet } F(v',h') \text{ of } \Phi(\xi)\colon (v',h')\in Q_{x_j,y})\\
&< \sum_{y=0}^{\infty}\sum_{j=0}^{m(y)}\PP\Big(\sup\limits_{\substack{(v,h)\in\partial \Psi(\xi),v\in (x_j\oplus [0,a]^{d-1})}} h > t+r^2/4+by\Big)\\
&<\sum_{y=0}^{\infty}m(y)p_{a,b}(\xi,y).
\end{aligned}
\end{equation}
It remains to estimate $m(y)$, which can be rewritten as
\begin{align*}
m(y)
&=\#\{x\in a\ZZ^{d-1}\colon (x\oplus [0,a]^{d-1})\cap (B_{R+\sqrt{r^2/4+b(y+1)}}\setminus B_{R+r-\sqrt{r^2/4+b(y+1)}}) \neq \varnothing\},
\end{align*}
where $\#\{\,\cdot\,\}$ stands for the cardinality of the set in brackets. Since all boxes intersecting $K(t,R,r,y)$ for fixed $y$ are included in the slightly extended set
\begin{align*}
\bigcup\limits_{1\leq i\leq m(y)}Q_{x_i,y}\subset \tilde K(t,R,r,y)&:=[t+r^2/4+by,t+r^2/4+b(y+1)]\\
&\qquad\qquad\times \Big(B_{R+\sqrt{r^2/4+b(y+1)}+\sqrt{d-1}a}\setminus B_{R+r-\sqrt{r^2/4+b(y+1)}-\sqrt{d-1}a}\Big),
\end{align*}
we can conclude that,
$$
m(y)\leq a^{-d+1}\Vol(\tilde K(t,R,r,y)).
$$
We will now distinguish the cases when $r \ge c_0R$ and $r< c_0R$ for some positive constant $c_0$. In the first case we use the simple bound
\begin{equation}\label{eq_19.10_4}
m(y)\leq \kappa_{d-1}a^{-d+1}\Big(R+\sqrt{r^2/4+b(y+1)}+\sqrt{d-1}a\Big)^{d-1}.
\end{equation}
In the second case we note that for
$$
b(y+1)\ge \big(R+r-\sqrt{d-1}a\big)^2-r^2/4
$$
we have $B_{R+r-\sqrt{r^2/4+b(y+1)}-\sqrt{d-1}a}=\varnothing$ and
\begin{equation}\label{eq_19.10_5}
\Vol(\tilde K(t,R,r,y))=\kappa_{d-1}\Big(R+\sqrt{r^2/4+b(y+1)}+\sqrt{d-1}a\Big)^{d-1}.
\end{equation}
Otherwise, remembering that $c_0R>r$ we obtain
\begin{equation}\label{eq_19.10_6}
\begin{aligned}
\Vol(&\tilde K(t,R,r,y))\\
&=\kappa_{d-1}\sum_{i=0}^{d-1}{d-1\choose i}\Big(\sqrt{r^2/4+b(y+1)}+\sqrt{d-1}a\Big)^{d-1-i}\big(R^i-(-1)^{d-1-i}(R+r)^i\big)\\
&\ll \sum_{i=0}^{d-2}{d-2\choose i}\Big(\sqrt{r^2/4+b(y+1)}+\sqrt{d-1}a\Big)^{d-1-i}(R^i+(R+r)^i)\\\
&\ll\Big(\sqrt{r^2/4+b(y+1)}+\sqrt{d-1}a\Big)\Big(R+\sqrt{r^2/4+b(y+1)}+\sqrt{d-1}a\Big)^{d-2},
\end{aligned}
\end{equation}
where in the second line we used the inequality ${d-1\choose i}\leq{1\over d-1}{d-2\choose i}$ and the fact that for $i=d-1$ we have $R^{d-1}-(R+r)^{d-1}\leq 0$, and in the third step we applied the inequality $R^i+(R+r)^i\leq (1+(1+c_0)^{d-2})R^i$ for any $0\leq i\leq d-2$.

To evaluate these terms in our situations, we first consider the case $\xi=\eta_{\beta}$. Take $t=0$ and $a={\sqrt{by+r^2/4}\over 2\sqrt{2(d-1)}}$. Then by estimate 2(a) of Lemma \ref{lm:4_3} we have $P_2(\eta_{\beta})=0$ and by estimate 1(a) of Lemma \ref{lm:4_3} for $r>0$ we obtain
\begin{equation}\label{eq_19.10_7}
p_{a,b}(\eta_\beta,y)\leq \exp\big(-C(by+r^2/4)^{{d+1\over 2}+\beta}\big).
\end{equation}
In case when $c_0R\leq r$ we take $b=1$ and we combine the estimate above together with \eqref{eq_19.10_1}, \eqref{eq_19.10_3} and \eqref{eq_19.10_4} to obtain
\begin{align*}
\PP(\cE(\eta_{\beta}))&\ll\sum\limits_{i=0}^{d-1}{d-1\choose i}R^{d-1-i}\sum_{y=0}^{\infty} (r^2/4+y)^{-{d-1\over 2}}(r^2/4+y+1)^{i\over 2}\exp\big(-C(y+r^2/4)^{{d+1\over 2}+\beta}\big).
\end{align*}
For $0\leq i\leq d-1$ consider the functions
$$
S_i(y):=(r^2/4+y)^{-{d-1\over 2}}(r^2/4+y+1)^{i\over 2}\exp\big(-C(y+r^2/4)^{{d+1\over 2}+\beta}\big),
$$
which are strictly decreasing for $y\ge 0$. Hence, substituting $s={r^2\over 4}+y$ and using
\begin{equation}\label{eq:22.06.21_5}
\sum_{y=0}^{\infty}S_i(y)\leq S_i(0)+\int_{0}^{\infty}S_i(y)\dint y,
\end{equation}
the sum in the last expression can be estimated as
\begin{equation}\label{eq:01.10.20}
\begin{aligned}
\PP(\cE(\eta_{\beta}))&\ll \sum\limits_{i=0}^{d-1}{d-1\choose i}R^{d-1-i}\Big(\int_{r^2/4}^{\infty} s^{-{d-1\over 2}}(s+1)^{i\over 2}\exp\big(-Cs^{{d+1\over 2}+\beta}\big)\dint s\\
&\qquad\qquad\qquad\qquad + r^{-d+1}(r^2+4)^{i\over 2}\exp\big(-Cr^{d+1+2\beta}\big)\Big).
\end{aligned}
\end{equation}
For $0\leq i\leq d-1$ and $Q>0$ consider now the integral
\begin{align*}
I_i(Q)&:=\int_{Q}^{\infty} s^{-{d-1\over 2}}(s+1)^{i\over 2}\exp\big(-Cs^{{d+1\over 2}+\beta}\big)\dint s.
\end{align*}
Since the function $(1+s^{-1})^{i/2}$, $0\leq i\leq d-1$ is decreasing for $s\in[Q,\infty)$ we have $(s+1)^{i\over 2}\leq s^{i\over 2}Q^{-{i\over 2}}(1+Q)^{i\over 2}$ and, thus,
\begin{align*}
I_i(Q)&\leq Q^{-{i\over 2}}(Q+1)^{i\over 2}\int_{Q}^{\infty} s^{-{d-i-1\over 2}}\exp\big(-Cs^{{d+1\over 2}+\beta}\big)\dint s.
\end{align*}
Making the change of variables $s^{d+1+2\beta\over 2}=Q^{{d+1\over 2}+\beta}t$ for $Q>0$ we arrive at
\begin{equation}\label{eq:22.06.21_1}
I_i(Q)\ll Q^{3-d\over 2}(Q+1)^{i\over 2}\int_{1}^{\infty} t^{-{2d-i-2+2\beta\over d+1+2\beta}}\exp\big(-CQ^{{d+1\over 2}+\beta}t\big)\dint t.
\end{equation}
For $0\leq i\leq d-3$ we have ${2d-i-2+2\beta\over d+1+2\beta} \ge 1$ and, thus,
\begin{align*}
I_i(Q)&\ll Q^{3-d\over 2}(Q+1)^{i\over 2}\int_{1}^{\infty} t^{-1}\exp\big(-CQ^{{d+1\over 2}+\beta}t\big)\dint t
\ll Q^{3-d\over 2}(Q+1)^{i\over 2}E_1(CQ^{{d+1\over 2}+\beta}),
\end{align*}
where $E_1(x)=\int_x^\infty t^{-1}e^{-t}\,\dint t$ is usual exponential integral, for which the upper bound
$$
E_1(x)\leq e^{-x}\log(1+1/x)<e^{-x}x^{-1},\qquad x>0,
$$
is well known, see \cite[Inequality (6.8.1)]{NIST}. Hence, for $0\leq i\leq d-3$ we obtain
\begin{equation}\label{eq:22.06.21_2}
I_i(Q)\ll Q^{1-d-\beta}(Q+1)^{i\over 2}\exp(-CQ^{{d+1\over 2}+\beta}).
\end{equation}
For $i=d-2$ and $i=d-1$ the integral in \eqref{eq:22.06.21_1} is proportional to $Q^{-{1\over 2}}\Gamma({1\over d+1+2\beta},CQ^{{d+1\over 2}+\beta})$ and $Q^{-1}\Gamma({2\over d+1+2\beta},CQ^{{d+1\over 2}+\beta})$ respectively, where $\Gamma(a,z)=\int_{z}^{\infty}t^{a-1}e^{-t}\dint t$ represents the incomplete gamma function. Using the recurrence relation
$$
\Gamma(a+n,z)=(a)_n\Gamma(a,z)+z^ae^{-z}\sum\limits_{k=0}^{n-1}{\Gamma(a+n)\over\Gamma(a+k+1)}z^k
$$
from \cite[Equation (8.8.9)]{NIST} and the estimate $z^{1-a}e^z\Gamma(a,z)\leq 1$ for $z>0$ and $0<a\leq 1$ from  \cite[Inequality (8.10.1)]{NIST} we get
\begin{equation}\label{eq:22.06.21_3}
\Gamma(a+n,z)\ll e^{-z}\max(z^{a+n-1},z^{a-1}).
\end{equation}
Thus, for $Q>1/4$, $i= d-2, d-1$, $d\ge 2$ and $\beta>-1$ we obtain
\begin{equation}\label{eq:22.06.21_4}
I_i(Q)\ll Q^{1-d-\beta}(Q+1)^{i\over 2}\exp(-CQ^{{d+1\over 2}+\beta}).
\end{equation}

Finally substituting \eqref{eq:22.06.21_2} and \eqref{eq:22.06.21_4} with $Q={r^2\over 4}$ into \eqref{eq:01.10.20} and taking  into accout that ${R\over r}\leq c_0$, $r\ge 1/3$, $d+2\beta\ge 0$ we obtain that for  some constants $c_1:=c_1(d,\beta)>0$, $c_2:=c_2(d, \beta, c_0)>0$,
\begin{align*}
\PP(\cE(\eta_{\beta}))&\ll (r^{-2d+2-2\beta}+r^{-d+1})(R+\sqrt{r^2+4})^{d-1}e^{-Cr^{d+1+2\beta}}\\
&\ll (r^{-d+1-2\beta}+1)\Big({R\over r}+\sqrt{1+{4\over r^2}}\Big)^{d-1}e^{-Cr^{d+1+2\beta}}\\
&\leq c_2re^{-c_1r^{d+1+2\beta}}.
\end{align*}
Consider the case $c_0R>r$. Taking $a={\sqrt{by+r^2/4}\over 2\sqrt{2(d-1)}}$ and choosing $b=1/8$ we denote by $Y$ the smallest integer $y$ such that
\begin{equation}\label{eq_20.10_1}
y\ge 8(R+r-{1\over8}\sqrt{2r^2+y})^2-2r^2-1
\end{equation}
holds. Then combining the estimate \eqref{eq_19.10_3}, \eqref{eq_19.10_5}---\eqref{eq_19.10_7} and inequality ${d-1\choose i}\leq {1\over d-1}{d-2\choose i}$ we get
\begin{align*}
\PP(\cE(\eta_{\beta}))&\ll \sum\limits_{i=0}^{d-1}{d-1\choose i}R^{d-1-i}\sum_{y=Y}^{\infty} (2r^2+y)^{-{d-1\over 2}}(2r^2+y+1)^{i\over 2}\exp\big(-C(y+2r^2)^{{d+1\over 2}+\beta}\big)\\
&\hspace{1cm}+\sum\limits_{i=0}^{d-2}{d-2\choose i}R^{d-2-i}\sum_{y=0}^{Y-1}(2r^2+y)^{-{d-1\over 2}}(2r^2+y+1)^{i+1\over 2}\exp\big(-C(y+2r^2)^{{d+1\over 2}+\beta}\big)\\
&\ll \sum\limits_{i=0}^{d-2}{d-2\choose i}R^{d-2-i}\sum_{y=0}^{\infty} (2r^2+y)^{-{d-1\over 2}}(2r^2+y+1)^{i+1\over 2}\exp\big(-C(y+2r^2)^{{d+1\over 2}+\beta}\big)\\
&\hspace{2cm}+R^{d-1}\sum_{y=Y}^{\infty} (2r^2+y)^{-{d-1\over 2}}\exp\big(-C(y+2r^2)^{{d+1\over 2}+\beta}\big).
\end{align*}
Using the integral estimate for the sum \eqref{eq:22.06.21_5} and inequalities \eqref{eq:22.06.21_2}, \eqref{eq:22.06.21_4} with $Q=2r^2$ we obtain
\begin{align*}
\PP(\cE(\eta_{\beta}))&\ll (r^{-2d+2-2\beta}+r^{-d+1})\sqrt{2r^2+1}(R+\sqrt{2r^2+1})^{d-2}e^{-Cr^{d+1+2\beta}}\\
&\hspace{2cm}+R^{d-1}\sum_{y=Y}^{\infty} (2r^2+y)^{-{d-1\over 2}}\exp\big(-C(y+2r^2)^{{d-1\over 2}+\beta+1}\big).
\end{align*}
Note that for $y=(R+r)^2-2r^2$ and $r\ge 1/3/2$ the inequality \eqref{eq_20.10_1} does not hold, that is why we have $Y>(R+r)^2-2r^2$ and substituting $s=2r^2+y$ we obtain
\begin{align*}
I(R,r)&:=R^{d-1}\sum_{y=Y}^{\infty} (2r^2+y)^{-{d-1\over 2}}\exp\big(-C(y+2r^2)^{{d+1\over 2}+\beta}\big)\\
&\ll R^{d-1}\int_{(R+r)^2}^{\infty} s^{-{d-1\over 2}}\exp\big(-Cs^{{d+1\over 2}+\beta}\big)\dint s+ R^{d-1}(R+r)^{-d+1}\exp\big(-C(R+r)^{d+1+2\beta}\big).
\end{align*}
Then by \eqref{eq:22.06.21_2} with $i=0$ and $Q=(R+r)^2$ and using the inequality $(R+r)^{d+1+2\beta}\ge R^{d+1+2\beta}+r^{d+1+2\beta}$ for $d\ge 2$, $\beta>-1$ and $R,r>0$ and estimate $Re^{-CR^{d+1+2\beta}}\ll 1$ we conclude
\begin{align*}
I(R,r)&\ll R^{d-1}((R+r)^{-2d+2-2\beta}+(R+r)^{-d+1})e^{-CR^{d+1+2\beta}}e^{-Cr^{d+1+2\beta}}\\
&\ll R^{d-2}(r^{-2d+2-2\beta}+r^{-d+1})e^{-Cr^{d+1+2\beta}}.
\end{align*}
Thus, since ${r\over R} < c_0$ and $r\ge 1/3$ we conclude that for some constants $\tilde{c_1}:=\tilde c_1(d,\beta,c_0)>0$ and $\tilde{c_2}:=\tilde c_2(d,\beta, c_0)>0$,
\begin{align*}
\PP(\cE(\eta_{\beta}))&\ll \big(r^{-2d+2-2\beta}+r^{-d+1}\big)\big(\sqrt{2r^2+1}(R+\sqrt{2r^2+1})^{d-2}+R^{d-2}\big)e^{-Cr^{d+1+2\beta}}\\
&\ll R^{d-2}r(r^{-2d+2-2\beta}+r^{-d+1})\Big(1+\sqrt{2\big(r/R\big)^2+R^{-2}}\Big)^{d-2}e^{-Cr^{d+1+2\beta}}\\
&\leq \tilde c_2 R^{d-2}r^{-d+3}e^{-\tilde c_1r^{d+1+2\beta}}.
\end{align*}
This completes the argument for the $\beta$-Delaunay tessellation.

Now consider the case $\xi=\eta^{\prime}_\beta$. Let $t=-r^2/4$, then for any $y\ge 0$ and $r>0$ we have $t+r^2/4+by\ge 0$ and by Lemma \ref{lm:4_3} estimate 1(b) we conclude that $p_{a,b}(\eta^{\prime}_{\beta},y)=0$ for all $y\ge 0$. On the other hand consider a partition of $\RR^{d-1}$ into boxes with side length ${c_0^{-1}r\over \sqrt{d-1}}$ and denote by $m$ the number of boxes intersecting the set $B_{R+r}\setminus B_R$. Then applying union bound and estimate 2(b) of Lemma \ref{lm:4_3} we obtain
\begin{align*}
P_2(\eta^{\prime}_{\beta})&\leq m\PP(\inf\limits_{\substack{(v,h)\in\partial \Psi(\xi),v\in B_{c_0^{-1}r}}} h < t)\leq m\Big(1-\exp\big(-Cr^{1+d-2\beta}\big)\Big)\ll m r^{-(2\beta-d-1)},
\end{align*}
where in the last step we used the inequality $1-e^{-x}\leq x$, $x\ge 0$. For $m$ we have the estimate
\begin{align*}
m\ll r^{-d+1}\Vol(B_{R+2c_0^{-1}r}\setminus B_{R-c_0^{-1}r}).
\end{align*}
For $R-c_0^{-1}r\leq 0$ we obtain
$$
P_2(\eta^{\prime}_{\beta})\ll \Big({R\over r}+2c_0^{-1}\Big)^{d-1}r^{-(2\beta-d-1)}\leq c_3r^{-(2\beta-d-1)}
$$
for some $c_3:=c_3(d,\beta,c_0)>0$. If $r< c_0R$, then
\begin{align*}
P_2(\eta^{\prime}_{\beta})&\ll r^{-d+1}\big((R+2c_0^{-1}r)^{d-1}-(R-c_0^{-1}r)^{d-1}\big)r^{-2\beta+d+1}\\
&\ll(R+2c_0^{-1}r)^{d-2}r^{-d+2}r^{-2\beta+d+1}\\
&\leq \tilde c_3 R^{d-2}r^{-2\beta+3}
\end{align*}
for some $\tilde c_3:=\tilde c_3(d,\beta,c_0)>0$. This together with \eqref{eq_19.10_1} yields the required bound for the $\beta'$-Delaunay tessellation.

Finally, consider the situation where $\xi=\zeta$. We again distinguish two cases when $r\ge c_0R$ and when $r<c_0R$ and as in the previous case we consider a partition of $\RR^{d-1}$ into boxes with side length    ${c_0^{-1}r\over \sqrt{d-1}}$. Let $r\ge c_0R$ and applying the estimate 2(c) of Lemma \ref{lm:4_3} with $t=-2(d-1)\ln r -{r^2\over 16}$ and the union bound we get
$$
P_2(\zeta)\ll\Big({R\over r}+2c_0^{-1}\Big)^{d-1}\big(1-\exp(-Ce^{-r^2/32})\big)\ll e^{-r^2/32}.
$$
By Lemma \ref{lm:4_3} estimate 1(c) with $2(d-1)a^2={r^2\over 32}$ and $b=1$ we get
\begin{equation}\label{eq_23.10_1}
p_{a,b}(\zeta, y)\ll \exp(-C(y+r^2/8)),
\end{equation}
where we used that $e^x\ge x+1$, $x\ge 0$. Together with the union bound, estimates \eqref{eq_19.10_3} and \eqref{eq_19.10_4}, and inequality $r\ge 1/3$ this shows that
\begin{align*}
P_1(\zeta)
&\ll r^{-d+1}e^{-{C\over 8}r^2}\sum_{i=0}^{d-1}{d-1\choose i}R^{d-1-i}\sum_{y=0}^{\infty}(r^2/4+y)^{i/2}e^{-Cy}.
\end{align*}
Since the function $S_i(y):=(r^2/4+y)^{i\over 2}e^{-Cy}$ is strictly decreasing in $y$ for any $0\leq i\leq d-1$ and $y\ge 0$, substituting $s={r^2\over 4}+y$ we have
\begin{align*}
P_1(\zeta)&\ll r^{-d+1}e^{-{C\over 4}r^2} \sum_{i=0}^{d-1}{d-1\choose i}R^{d-1-i}\Big( e^{Cr^2/4}\int_{Cr^2/4}^{\infty}s^{{i\over 2}}e^{-s}\dint s+r^{i}\Big)\\
&\ll e^{-{C\over 4}r^2}\sum_{i=0}^{d-1}{d-1\choose i}R^{d-1-i}r^{-d+1+i}\\
&\ll e^{-Cr^2},
\end{align*}
where in the second line we used the inequality \eqref{eq:22.06.21_3} for incomplete Gamma function $\Gamma(i/2+1,r^2/4)$, $r\ge 1/3$, and in the third line we applied the estimate ${R\over r}\leq c_0^{-1}$. Together with the bound for $P_2(\zeta)$ provided above we get, for some constants $c_5:=c_5(d,c_0)>0$ and $c_4:=c_4(d,c_0)>0$,
$$
P(\cE(\zeta))<c_5e^{-c_4r^2}.
$$
When $c_0R>r$, $r\ge 1/3$ we again choose $t=-2(d-1)r -{r^2\over 16}$ and by the estimate 2(c) of Lemma \ref{lm:4_3}  together with union bound we obtain
\begin{align*}
P_2(\zeta)&\ll R^{d-2}r^{-d+2} e^{-r^2/32},
\end{align*}
where we proceeded as in the analogous situation for $\eta^{\prime}_{\beta}$. Take $2(d-1)a^2={r^2\over 32}$, $b=1$ and denote by $Y$ the smallest integer $y$ such that $y\ge (R+7r/8)^2-r^2/4-1$. It is clear that $Y\ge {1\over 2}(R+r)^2-r^2/4$. Combining the estimates \eqref{eq_19.10_3}, \eqref{eq_19.10_5}, \eqref{eq_19.10_6} and \eqref{eq_23.10_1} and proceeding as in the case of $\eta_{\beta}$ we obtain
\begin{align*}
P_1(\zeta)&\ll r^{-d+1}e^{-Cr^2}\sum\limits_{i=0}^{d-2}{d-2\choose i}R^{d-2-i}\sum_{y=0}^{\infty} (r^2/4+y)^{i+1\over 2}e^{-Cy}+R^{d-1}r^{-d+1}e^{-Cr^2}\sum_{y=Y}^{\infty}e^{-Cy}.
\end{align*}
Using integral bound for the sum and applying \eqref{eq:22.06.21_3} and $(R+r)^2\ge R^2+r^2$ we get
\begin{align*}
P_1(\zeta)&\ll R^{d-2}r^{-d+2}e^{-Cr^2}+R^{d-1}e^{-CR^2}r^{-d+1}e^{-Cr^2}\ll R^{d-2}r^{-d+2}e^{-C r^2},
\end{align*}
Together with the estimate for $P_2(\zeta)$ we obtain for some constants $\tilde c_5:=\tilde c_5(d,c_0)>0$ and $\tilde c_4:=\tilde c_4(d,c_0)>0$,
$$
P(\cE(\zeta))<\tilde c_5R^{d-2}r^{-d+2}e^{-\tilde c_4r^2}.
$$
This completes the argument.
\end{proof}

The previous lemma provides a stabilization property for the $\beta$-, the $\beta'$- and the Gaussian Delaunay tessellation \textit{inside} a ball with fixed radius. We need a similar property also for these tessellations \textit{outside} such balls. The next lemma provides a corresponding result.

\begin{lemma}\label{lm:4_5}
\begin{enumerate}
\item For any $R'-r'>0$, $r'\ge 1/3$, $d\ge 2$, $\beta>-1$ and for any constant $c_0>0$ we have
\begin{align*}
\PP(\widetilde{\cD}_{\beta}\cap (\RR^{d-1}\setminus B_{R'}) &\text{ is not determined by particles } (v,h)\in\partial \Psi(\zeta_{\beta})\text{ with } v\in (\RR^{d-1}\setminus B_{R'-r'}))\\
&< \begin{cases}
c'_2r'e^{-c'_1(r')^{d+1+2\beta}}\quad &:\quad r'\ge c_0(R'-r'),\\
\tilde c'_2 (R'-r')^{d-2}(r')^{-d+3}e^{-\tilde c'_1(r')^{d+1+2\beta}}\quad &:\quad r'< c_0(R'-r'),
\end{cases}
\end{align*}
for some positive constants $c'_1,\tilde c'_1, c'_2, \tilde c'_2$, depending on $d$, $\beta$ and $c_0$ only.
\item For any $R'-r'>0$, $r'>0$, $d\ge 2$ and $\beta >(d+1)/2$, and for any constant $c_0>0$ we have
\begin{align*}
\PP(\widetilde{\cD}^{\prime}_{\beta}\cap (\RR^{d-1}\setminus B_{R'}) &\text{ is not  determined by particles } (v,h)\in\partial \Psi(\zeta_{\beta}^{\prime})\text{ with } v\in (\RR^{d-1}\setminus B_{R'-r'}))\\
&< \begin{cases}
c'_3(r')^{-(2\beta-d-1)}\quad &:\quad r'\ge c_0(R'-r'),\\
\tilde c'_3 (R'-r')^{d-2}(r')^{-2\beta+3}\quad &:\quad r'< c_0(R'-r'),
\end{cases}
\end{align*}
for some positive constants $c'_3,\tilde c'_3$, depending on $d$, $\beta$ and $c_0$ only.
\item For any $R'-r'>0$, $r'\ge 1/3$, and $d\ge 2$, and for any constant $c_0>0$ we have
\begin{align*}
\PP(\widetilde{\cD}\cap (\RR^{d-1}\setminus B_{R'}) &\text{ is not determined by particles } (v,h)\in\partial \Psi(\zeta)\text{ with }v\in (\RR^{d-1}\setminus B_{R'-r'}))\\
&< \begin{cases}
c'_5e^{-c'_4(r')^2}\quad &:\quad r'\ge c_0(R'-r'),\\
\tilde c'_5 (R'-r')^{d-2}(r')^{-d+2}e^{-\tilde c'_4(r')^2}\quad &:\quad r'< c_0(R'-r'),
\end{cases}
\end{align*}
for some positive constants $c'_4,\tilde c'_4, c'_5, \tilde c'_5$, depending on $d$ and $c_0$ only.
\end{enumerate}
\end{lemma}

\begin{proof}
The proof is very similar to the proof of Lemma \ref{lm:4_4}, but for completeness we provide some details. In what follows, we will write $\xi$ for one of the Poisson point processes $\zeta$, $\eta_{\beta}$ or $\eta^{\prime}_{\beta}$. 
 For $R'\ge 1/3$ and $1/3\leq r'\leq R'$ consider the event
\begin{align*}
\cE(\xi) &:=\{\widetilde{\cD}(\xi)\cap (\RR^{d-1}\setminus B_{R'}) \text{ is not determined by particles } (v,h)\in\partial \Psi(\xi)\\
&\hspace{8cm}\text{ with }v\in (\RR^{d-1}\setminus B_{R'-r'})\}.
\end{align*}
By the law of total probability we have that $\PP(\cE(\xi))\leq P_1(\xi)+P_2(\xi)$, where
\begin{align*}
P_1(\xi)&:= \PP\big(\cE(\xi) \, \big| \, \inf\limits_{\substack{(v,h)\in\partial \Psi(\xi),v\in \cl(B_{R'}\setminus B_{R'-r'})}} h > t\big),\\
P_2(\xi)&:=\PP\big(\inf\limits_{\substack{(v,h)\in\partial \Psi(\xi),v\in \cl(B_{R'}\setminus B_{R'-r'})}} h < t\big).
\end{align*}

According to the construction $\widetilde{\cD}(\xi)\cap (\RR^{d-1}\setminus B_{R'})$ is determined as soon as we know the location of all vertices of paraboloid facets of $\Phi(\xi)$ hitting the set $(\RR^{d-1}\setminus B_{R'})\times \RR$ and the event $\cE(\xi)$ occurs if and only if there is paraboloid facet of $\Phi(\xi)$ hitting the set $(\RR^{d-1}\setminus B_{R'})\times \RR$ and having a vertex $(v,h)$ with $v\not \in (\RR^{d-1}\setminus B_{R'-r'})$. Let $\Pi_{-}(v',h')$ be a paraboloid such that $F(v',h'):=\Pi_{-}(v',h')\cap \Phi(\xi)$ is a paraboloid facet of $\Phi(\xi)$. Then conditioning on
$$
\inf\limits_{\substack{(v,h)\in\partial \Psi(\xi),v\in \cl(B_{R'}\setminus B_{R'-r'})}} h < t,
$$
we have
\begin{align*}
P_1(\xi)&\leq \PP(\exists\text{ a paraboloid facet } F(v',h') { of } \Phi(\xi)\colon\\
&\hspace{1cm} B_{\sqrt{h'-t}}(v')\cap (\RR^{d-1}\setminus B_{R'})\neq \varnothing, B_{\sqrt{h'-t}}(v')\cap B_{R'-r'}=\varnothing),
\end{align*}
where the arguments are the same as in the proof of Lemma \ref{lm:4_4}. The rest now follows as in the proof of Lemma \ref{lm:4_4} with $R=R'-r'$ and $r=r'$.
\end{proof}

\section{Absolute regularity for $\beta^{(')}$- and Gaussian Delaunay tessellations}\label{sec:AbsoluteRegularity}

In this section we establish absolute regularity for the $\beta$-Delaunay tessellation $\widetilde{\cD}_\beta$, the $\beta'$-Delaunay tessellation $\widetilde{\cD}_\beta'$ and the Gaussian-Delaunay tessellation $\widetilde{\cD}$. More specifically, we prove tight bounds on the absolute regularity coefficients $\mathscr{B}(a,b)$ for $a>0$, $b-a\ge 1$, which will be of importance in the next section when we deal with central limit theorems.

\begin{theorem}\label{thm:BetaEstimates}
\begin{itemize}
\item[(i)] Let $\beta> -1$. Then the $\beta$-Delaunay tessellation $\widetilde{\cD}_\beta$ is absolutely regular and, moreover, for any $a>0$, $b-a\ge 1$ and $c_0>0$ we have
$$
\mathscr{B}(a,b)\leq\begin{cases} c_7(d,\beta,c_0)(b-a)e^{-c_6(d,\beta)(b-a)^{d+1+2\beta}} &: b-a>c_0a\\
\tilde c_7(d,\beta,c_0)a^{d-2}(b-a)^{-d+3}e^{-\tilde c_6(d,\beta)(b-a)^{d+1+2\beta}} &: b-a\leq c_0a,
\end{cases}
$$
where $c_6(d,\beta)$, $\tilde c_6(d,\beta)$, $c_7(d,\beta,c_0)$ and $\tilde c_7(d,\beta,c_0)$ are positive constants only depending on $d$ and $\beta$ and additionally on $c_0$ in case of $c_7, \tilde c_7$.
\item[(ii)] Let $\beta>(d+1)/2$. Then the $\beta'$-Delaunay tessellation $\widetilde{\cD}_\beta'$ is absolutely regular and, moreover, for any $b>a>0$ and $c_0>0$ we have
$$
\mathscr{B}(a,b)\leq \begin{cases} c_8(d,\beta,c_0)(b-a)^{-(2\beta-d-1)} &: b-a>c_0a\\
\tilde c_8(d,\beta,c_0)a^{d-2}(b-a)^{-2\beta+3} &: b-a\leq c_0a,
\end{cases}
$$
where $c_8(d,\beta,c_0)$ and $\tilde c_8(d,\beta,c_0)$ are positive constants only depending on $d$ and $\beta$ and $c_0$.
\item[(iii)] The Gaussian-Delaunay tessellation $\widetilde{\cD}$ is absolutely regular and, moreover, for any $a>0$, $b-a\ge 1$ and $c_0>0$ we have
$$
\mathscr{B}(a,b)\leq \begin{cases}
c_{10}(d,c_0)e^{-c_9(d)(b-a)^2} &: b-a>c_0a\\
\tilde c_{10}(d,c_0)a^{d-2}(b-a)^{-d+2}e^{-\tilde c_9(d)(b-a)^2} &: b-a\leq c_0a
\end{cases}
$$
where $c_9(d)$, $\tilde c_9(d)$, $c_{10}(d,c_0)$ and $\tilde c_{10}(d,c_0)$ are positive constants only depending on $d$ and additionally on $c_0$ in case of $c_{10}, \tilde c_{10}$.
\end{itemize}
\end{theorem}

\begin{proof}
Let us write $\xi$ for one of the Poisson point processes $\eta_{\beta}$, $\eta^{\prime}_{\beta}$ and $\zeta$. 
 For fixed $0<a<b$ we consider the following two events
\begin{figure}[t]
\begin{center}
\includegraphics[width=0.65\columnwidth]{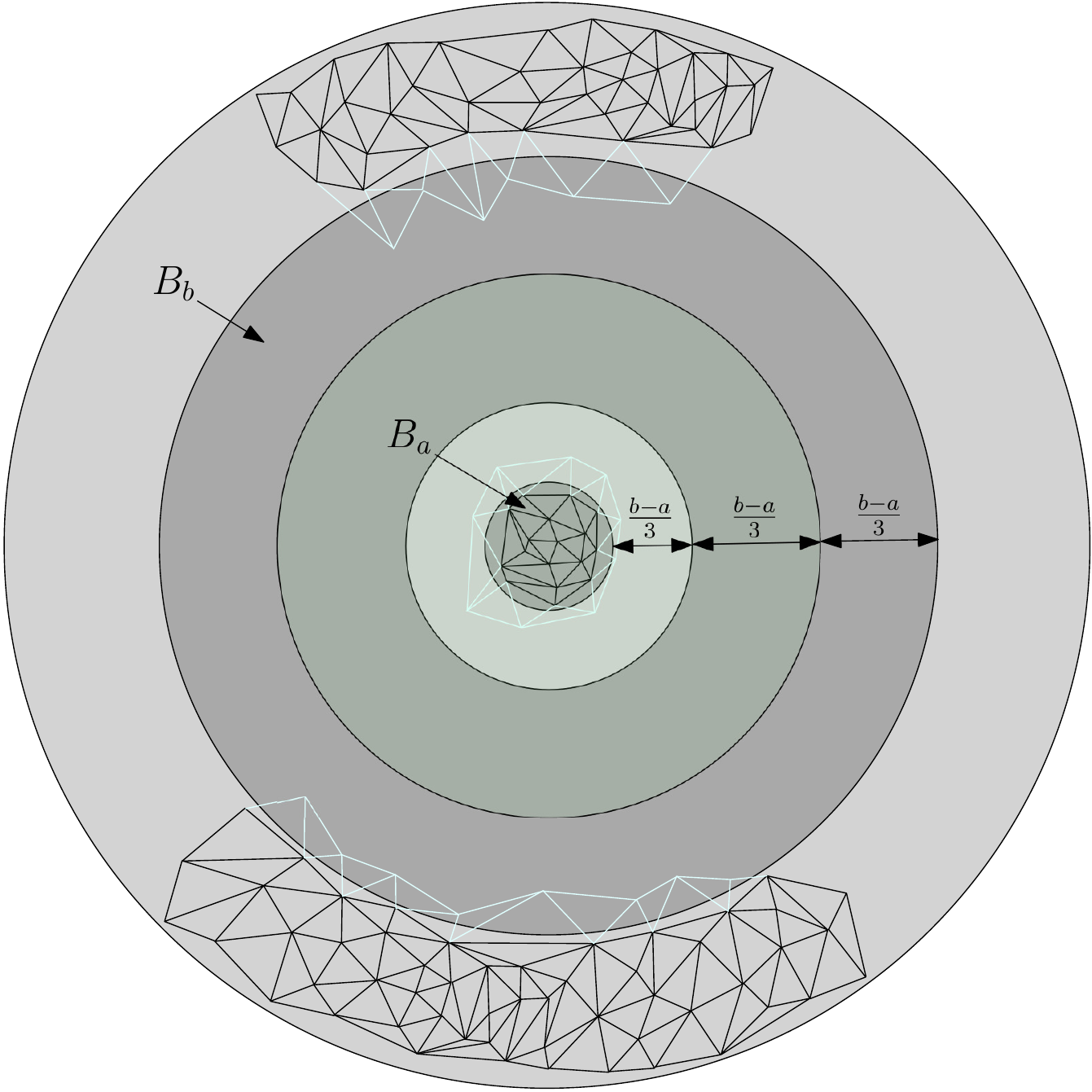}
\end{center}
\caption{Illustration of the events $\cE_0$ and $\cE_1$. Indicated is the tessellation $\cD(\xi)$ inside the ball $B_a$ (which under $\cE_0$ is determined by particles of $\partial\Psi(\xi)$ with spatial coordinate in $B_{a+(b-a)/3}$) and outside the ball $B_{b}$ (which under $\cE_1$ is determined by particles of $\partial\Psi(\xi)$ with spatial coordinate outside of $B_{a+2(b-a)/3}$).}
\label{fig:StabilizationBaBb}
\end{figure}
\begin{align*}
\cE_0&:=\{\cD(\xi)\cap B_a \text{ is determined by particles } (v,h)\in\partial \Psi(\xi) \text{ with }v\in B_{a+(b-a)/3}\},\\
\cE_1&:=\{\cD(\xi)\cap (\RR^{d-1}\setminus B_{b}) \text{ is determined by particles } (v,h)\in\partial \Psi(\xi)\\
&\hspace{8cm} \text{ with }v\in (\RR^{d-1}\setminus B_{a+2(b-a)/3})\},
\end{align*}
which are illustrated in Figure \ref{fig:StabilizationBaBb}. Using the law of total probability, for any $\cA^{\prime}\in\frak{F}^{d-1}_{-b}$ and any $\cA\in\frak{F}_a^{d-1}$ we have that
\begin{align*}
|\PP(\cD(\xi)&\in\cA^{\prime}|\cD(\xi)\in\cA)-\PP(\cD(\xi)\in\cA^{\prime})|\\
&\leq |\PP(\cD(\xi)\in\cA^{\prime}| \{\cD(\xi)\in\cA\}\cap \cE_0\cap\cE_1)-\PP(\cD(\xi)\in\cA^{\prime}|\cE_0\cap\cE_1)|\PP(\cE_0\cap\cE_1)\\
&\qquad+|\PP(\cD(\xi)\in\cA^{\prime}| \{\cD(\xi)\in\cA\}\cap (\cE_0\cap\cE_1)^{c})-\PP(\cD(\xi)\in\cA^{\prime}|(\cE_0\cap\cE_1)^c)|\PP((\cE_0\cap\cE_1)^c)\\
&\leq\PP(\cE_0^c)+\PP(\cE_1^c),
\end{align*}
since the events $\{\cD(\xi)\in\cA\}$ and $\{\cD(\xi)\in\cA^{\prime}\}$ are independent conditionally on $\cE_0\cap\cE_1$. Now, we apply Lemma \ref{lm:betamixing} with $Z:=\cD(\xi)$ and $\cB=\cE_0$. This yields the estimate
$$
\mathscr{B}(a,b)\leq 2\PP(\cE_0^c)+\PP(\cE_1^c).
$$
The probabilities $\PP(\cE_0^c)$ and $\PP(\cE_1^c)$ can be determined with the help of Lemma \ref{lm:4_4} and Lemma \ref{lm:4_5}. This leads to the following estimates (here, $c_6,\tilde c_6, c_7,\tilde c_7,\ldots$ are positive constants only depending on the respective parameter indicated in brackets):
\begin{enumerate}
\item If $\xi=\eta_{\beta}$ with $\beta>-1$, then we take $R=a$ and $r=(b-a)/3$ in Lemma \ref{lm:4_4} and $R' = b$, $r'=(b-a)/3$ in Lemma \ref{lm:4_5}. If $b-a > c_0a$, then $r> 3c_0R$ and
$$
{c_0\over 3+2c_0}(R'-r')={2c_0\over 3(3+2c_0)}b+{c_0\over 3(3+2c_0)}a={b-a\over 3}-{b\over 3+2c_0}+{a(1+c_0)\over 3+2c_0}\leq {b-a\over 3}=r'.
$$
Thus, we get
\begin{align*}
\mathscr{B}(a,b)&\leq c_7(d,\beta,c_0)(b-a)e^{-c_6(d,\beta)(b-a)^{d+1+2\beta}}.
\end{align*}
Clearly, this implies that
$$
\lim\limits_{b\to\infty}\mathscr{B}(a,b)=0
$$
for any fixed $a>0$, showing that the $\beta$-Delaunay tessellation $\cD_{\beta}$ is absolutely regular for any $\beta>-1$.

On the other hand, if $b-a\leq c_0a$ and $b-a\ge 1$ then $r<3c_0R$ and ${c_0(R'-r')\over 3+2c_0}>r'$ and we obtain
\begin{align*}
\mathscr{B}(a,b)&\leq \tilde c_2(d,\beta,c_0)3^{d-3}a^{d-2}(b-a)^{-d+3}e^{-\tilde c_1(d,\beta)((b-a)/3)^{d+1+2\beta}}\\
&\qquad\qquad+\tilde c'_2(d,\beta,c_0)3^{-1}(2b+a)^{d-2}(b-a)^{-d+3}e^{-\tilde c'_1(d,\beta)((b-a)/3)^{d+1+2\beta}}\\
&\leq \tilde c_2(d,\beta,c_0)3^{d-3}a^{d-2}(b-a)^{-d+3}e^{-\tilde c_1(d,\beta)((b-a)/3)^{d+1+2\beta}}\\
&\qquad\qquad+\tilde c'_2(d,\beta,c_0)3^{-1}(3a+2c_0a)^{d-2}(b-a)^{-d+3}e^{-\tilde c'_1(d,\beta)((b-a)/3)^{d+1+2\beta}}\\
&\leq \tilde c_7(d,\beta,c_0)a^{d-2}(b-a)^{-d+3}e^{-\tilde c_6(d,\beta)(b-a)^{d+1+2\beta}}.
\end{align*}

\item If $\xi=\eta^{\prime}_{\beta}$ with $\beta>(d+1)/2$ we apply Lemma \ref{lm:4_4} with $R=a$ and $r=(b-a)/3$ and Lemma \ref{lm:4_5} with $R'=b$ and $r'=(b-a)/3$. Then, as before, for $b-a>c_0a$ we have
\begin{align*}
\mathscr{B}(a,b)&\leq c_8(d,\beta, c_0)(b-a)^{-(2\beta-d-1)}.
\end{align*}
From this we deduce that
$$
\lim\limits_{b\to\infty}\mathscr{B}(a,b)=\lim_{b\to\infty}\tilde c_8(d,\beta,c_0)b^{-(2\beta-1-d)}=0,
$$
for any fixed $a>0$, which proves that for any $\beta>(d+1)/2$ the $\beta'$-Delaunay tessellation is absolutely regular. In addition, if $b-a\leq c_0a$ then
\begin{align*}
\mathscr{B}(a,b)&\leq \tilde c_3(d,\beta,c_0)3^{2\beta-3}a^{d-2}(b-a)^{-2\beta+3}+\tilde c_3'(d,\beta,c_0)3^{2\beta-3}(2/3b+1/3a)^{d-2}(b-a)^{-2\beta+3}\\
&\leq \tilde c_8(d,\beta,c_0)a^{d-2}(b-a)^{-2\beta+3},
\end{align*}
where we used the fact that $b\leq a(1+c_0)$.

\item Finally, if $\xi=\zeta$ we again use Lemma \ref{lm:4_4} with $R=a$ and $r=(b-a)/3$ and Lemma \ref{lm:4_5} with $R' = b$, $r'=(b-a)/3$. For $b-a>c_0a$ we obtain
\begin{align*}
\mathscr{B}(a,b)&\leq c_{10}(d,c_0)e^{-c_9(d)(b-a)^2}.
\end{align*}
Thus, $\lim_{b\to\infty}\mathscr{B}(a,b)=0$ for any $a>0$ and the Gaussian-Delaunay tessellation is absolutely regular.
For $b-a\leq c_0a$ and $b-a\ge 1$ be get
\begin{align*}
\mathscr{B}(a,b)&\leq \tilde c_{10}(d,c_0)a^{d-2}(b-a)^{-d+2}e^{-\tilde c_9(d)(b-a)^2},
\end{align*}
where we applied the same inequalities as in the previous computations.
\end{enumerate}
The proof is thus complete.
\end{proof}

\section{Central limit theorems}\label{sec:CLTs}

The bounds on the absolute regularity coefficients established in the previous section can be used to deduce central limit theorems for combinatorial and geometric parameters of $\beta$- and Gaussian Delaunay tessellations; for the $\beta'$-Delaunay tessellation see Remark \ref{rem:CLTbetaPrime} below. We start with the  number of $k$-dimensional faces. To be formal, let $c(\,\cdot\,)$ be a so-called centre function which associates with a (potentially lower-dimensional) simplex $S\subset\RR^{d-1}$ a point $c(S)\in\RR^{d-1}$, the centre of $S$, in such a way that $c(S+v)=c(S)$ for all $v\in\RR^{d-1}$. There are many choices for such centre functions, for example $c(S)$ could be the centre of gravity of $S$ or the centre of the smallest ball containing $S$. For us it will turn out to be convenient to take for $c(S)$ the lexicographically smallest vertex of $S$. In what follows we understand $c(S)$ in exactly this way. Now, for a stationary random tessellation $\cT$ (identified with its skeleton, that is, the random closed set generated by the union of all cell boundaries) in $\RR^{d-1}$ denote by $X_{\cT,k}$ the stationary point process formed by all centres (that is, lexicographically smallest vertices) of $k$-dimensional faces of $\cT$, which are almost surely simplices of dimension $k$. In particular, for sets $B\subset\RR^{d-1}$, $X_{\cT,0}(B)$ counts the number of vertices of $\cT$ in $B$, $X_{\cT,1}(B)$ is the number of lexicographically smallest edge endpoints  in $B$ of edges of $\cT$ and $X_{\cT,d-1}(B)$ is the same as the number of cells of $\cT$ whose lexicographically smallest vertex is located in $B$.

In what follows we shall use the notation $I_n:=[-n,n]^{d-1}$ for $n\in\NN$, and let $\lambda_{\cT,k}:=\EE X_{\cT,k}([0,1]^{d-1})$ be the intensity of the stationary point process $X_{\cT,k}$. We note that $\EE X_{\cT,k}(I_n)=(2n)^{d-1}\lambda_{\cT,k}$. For the $\beta$- (and $\beta'$-) Delaunay tessellation the explicit values of $\lambda_{\widetilde{\cD}_\beta,k}$ (and $\lambda_{\widetilde{\cD}_\beta',k}$) were determined in \cite[Theorem 6.4]{GKT1} and for the Gaussian-Delaunay tessellation see \cite[Corollary 5.9]{GKT2}.

\begin{theorem}\label{thm:CLT}
Let $\cT$ be one of the tessellations $\widetilde{\cD}_\beta=\cD(\eta_\beta)$ for $\beta>-1$ or $\widetilde{\cD}=\cD(\zeta)$. Further, fix $k\in\{0,1,\ldots,d-1\}$. Then the limit
\begin{equation}\label{eq:CLTVariance}
s_k^2 := \lim_{n\to\infty}(2n)^{-(d-1)}\Var X_{\cT,k}(I_n)
\end{equation}
exists and is strictly positive. Moreover, we have the convergence in distribution
\begin{equation}\label{eq:CLTConvergence}
{X_{\cT,k}(I_n)-(2n)^{d-1}\lambda_{\cT,k}\over (2n)^{d-1\over 2}}\overset{d}{\longrightarrow} G,
\end{equation}
as $n\to\infty$, where $G\sim\mathcal{N}(0,s_k^2)$ is a centred Gaussian random variable with variance $s_k^2$.
\end{theorem}

\begin{remark}
The central limit theorem in Theorem \ref{thm:CLT} continues to holds if the sequence of cubes $I_n$ is replaced by an arbitrary sequence of compact convex sets $W_n\subset\RR^{d-1}$ with interior points having the property that $r(W_n)\to\infty$, as $n\to\infty$, where $r(W_n)$ denotes the inradius of $W_n$. The same comment applies to Theorem \ref{thm:CLTkvolume} below.
\end{remark}

We prepare the proof of Theorem \ref{thm:CLT} with the following auxiliary result, which is more of technical nature.

\begin{lemma}\label{lem:MomentsForCLT}
	For any $k\in\{0,1,\ldots,d-1\}$ the random variables $X_{\widetilde{\cD}_\beta,k}(I_1)$ and $X_{\widetilde{\cD},k}(I_1)$ have finite moments of all orders.
\end{lemma}
\begin{proof}
	Fix some $R>0$. We start by proving that for any $m\in\NN$,
	$$
	\EE X_{\widetilde{\cD}_\beta,0}(B_R)^m<\infty\qquad\text{and}\qquad\EE X_{\widetilde{\cD},0}(B_R)^m<\infty.
	$$
	By construction of the tessellations $\widetilde{\cD}_\beta$ and $\widetilde{\cD}$, $X_{\widetilde{\cD}_\beta,0}(B_R)$ is the same as the number of vertices in $B_R\times\RR_+$ of the paraboloid hull process associated with the Poisson point process $\eta_\beta$ and $X_{\widetilde{\cD},0}(B_R)$ is the number of vertices in $B_R\times\RR$ of the paraboloid hull process generated by $\zeta$.
	
	Let $H_\beta>0$ be the height of the cylinder in $\RR^{d-1}\times\RR_+$ with base $B_R$ which contains all vertices of $\Phi(\eta_\beta)\cap (B_R\times\RR_+)$. Similarly, let $H\in\RR$ be the height of the cylinder $\RR^{d-1}\times\RR$ with base $B_R$ which contains all vertices of $\Phi(\zeta)\cap (B_R\times\RR)$. Then the two inequalities
	$$
	X_{\widetilde{\cD}_\beta,0}(B_R) \leq \eta_\beta(B_R\times[0,H_\beta])\qquad\text{and}\qquad
	X_{\widetilde{\cD}_\beta,0}(B_R) \leq \zeta(B_R\times(-\infty,H])
	$$
	hold with probability one. Note that the dominating random variables are Poisson distributed with parameter $c_{d,\beta,R}\int_0^{H_\beta}h^\beta\,\dint h$ and $c_{d,R}\int_{-\infty}^He^{h/2}\,\dint h$, respectively, where $c_{d,\beta,R}>0$ is some constant only depending on $d,\beta,R$ and $c_{d,R}>0$ is a constant only depending on $d$ and $R$; the explicit values are not relevant for our purposes. As a consequence,
	$$
	\EE X_{\widetilde{\cD}_\beta,0}(B_R)^m \leq c_{d,\beta,R,m}\,\EE\Big(\int_0^{H_\beta}h^\beta\,\dint h\Big)^m \leq \hat c_{d,\beta,R,m}\,\EE[H_\beta^{(\beta+1)m}]
	$$
	and
	$$
	\EE X_{\widetilde{\cD},0}(B_R)^m \leq c_{d,R,m}\,\EE\Big(\int_{-\infty}^He^{h/2}\,\dint h\Big)^m \leq \hat c_{d,R,m}\,\EE[e^{Hm/2}]
	$$
	for constants $c_{d,\beta,R,m},\hat c_{d,\beta,R,m}>0$ and $c_{d,R,m},\hat c_{d,R,m}>0$ depending only on $d,\beta,R$ and $m$ or $d,R$ and $m$, respectively. However, according to Lemma \ref{lm:4_3} the random variables $H_\beta$ have sub-exponential tails, which especially implies that $\EE H_\beta^p<\infty$ for any $p>0$, see \cite[Proposition 2.7.1]{VershyninBook}. Moreover, by the same lemma, the random variable $H$ has double-exponential tails, which yields that $\EE e^{p H}<\infty$ for any $p\in\RR$. We thus conclude that
	\begin{align}\label{eq:MomBdPrf0}
	\EE X_{\widetilde{\cD}_\beta,0}(B_R)^m < \infty \qquad\text{and}\qquad\EE X_{\widetilde{\cD},0}(B_R)^m < \infty
	\end{align}
	for any $m\in\NN$. This readily implies the result for $k=0$.
	
	For $k\in\{1,\ldots,d-1\}$ we put $\varrho:=\sqrt{d-1}$ and start by observing that
	$$
		\EE X_{\widetilde{\cD}_\beta,k}(I_1)^m \leq \EE X_{\widetilde{\cD}_\beta,k}(B_\varrho)^m\qquad\text{and}\qquad	\EE X_{\widetilde{\cD},k}(I_1)^m \leq \EE X_{\widetilde{\cD},k}(B_\varrho)^m
	$$
	Next, we observe that $X_{\widetilde{\cD}_\beta,k}(B_\varrho)$ is clearly dominated by summing over all vertices of $\widetilde{\cD}_\beta$ in $B_\varrho$ the number of adjacent $k$-faces, and similarly for the Gaussian-Delaunay tessellation $\widetilde{\cD}$. Formally, for a vertex $v$ of a tessellation $\cT$ we write $f_k(v)$ for the the number of $k$-faces of $\cT$ that are adjacent to $v$, meaning that they have $v$ as a vertex. We also define $f_k(v)=0$ if $v\in\RR^{d-1}$ is not a vertex of $\cT$. Then, almost surely it holds that
	$$
	X_{\widetilde{\cD}_\beta,k}(B_\varrho) \leq \sum_{v\in X_{\widetilde{\cD}_\beta,0}(B_\varrho)}f_k(v)\qquad\text{and}\qquad X_{\widetilde{\cD},k}(B_\varrho) \leq \sum_{v\in X_{\widetilde{\cD},0}(B_\varrho)}f_k(v)
	$$
	and we conclude, using the Cauchy-Schwarz inequality, that
	\begin{align}
	\nonumber\EE X_{\widetilde{\cD}_\beta,k}(B_\varrho)^m &\leq \EE\Big(\sum_{v\in X_{\widetilde{\cD}_\beta,0}(B_\varrho)}f_k(v)\Big)^m\\
	&\nonumber\leq\EE\big(X_{\widetilde{\cD}_\beta,0}(B_\varrho)\sup_{v\in B_\varrho}f_k(v)\big)^m\\
	&\nonumber\leq\EE\big(X_{\widetilde{\cD}_\beta,0}(B_\varrho)\sup_{v\in \RR^{d-1}}f_k(v)\big)^m\\
	&\leq\sqrt{\EE X_{\widetilde{\cD}_\beta,0}(B_\varrho)^{2m}}\sqrt{\sup_{v\in \RR^{d-1}}\EE f_k(v)^{2m}}.\label{eq:19-01-21a}
	\end{align}
	The first term in the last line is finite according to \eqref{eq:MomBdPrf0} with $R=\varrho$ there, and the second term is finite by the first part of \cite[Lemma 7.1]{CalkaSchreiberYukich}, implying altogether that $\EE X_{\widetilde{\cD}_\beta,k}(I_1)^m\leq\EE X_{\widetilde{\cD}_\beta,k}(B_\varrho)^m<\infty$ (the result from \cite{CalkaSchreiberYukich} we used actually deals with the number of $k$-dimensional paraboloid faces of the paraboloid hull process, but by definition of the $\beta$-Delaunay tessellation there is a one-to-one correspondence between these faces and the $k$-faces of the tessellation $\widetilde{\cD}_\beta$). Similarly, for the Gaussian-Delaunay tessellation we have that
	$$
	\EE X_{\widetilde{\cD},k}(I_1)^m\leq\EE X_{\widetilde{\cD},k}(B_\varrho)^m \leq \sqrt{\EE X_{\widetilde{\cD},0}(B_\varrho)^{2m}}\sqrt{\sup_{v\in \RR^{d-1}}\EE f_k(v)^{2m}}.
	$$
	In this situation the first term is finite by \eqref{eq:MomBdPrf0} again with $R=\varrho$ there, while the second one is finite by \cite[Lemma 4.4]{CalkaYukichGaussian}, applied with $\lambda=\infty$ (again subject to the translation between paraboloid $k$-faces of the paraboloid hull process and $k$-faces of the tessellation $\widetilde{\cD}$).
\end{proof}

In the next lemma we will show that the asymptotic variances $s_k^2$, defined in \eqref{eq:CLTVariance} are strictly positive. Our proof is inspired by similar results for random polytopes, first established in \cite{ReitznerCLT}.

\begin{lemma}\label{lem:sk>0}
Under conditions of Theorem \ref{thm:CLT} we have that $s_k^2>0$ for all $k\in\{0,1,\ldots,d-1\}$.
\end{lemma}

\begin{proof}
We need to show that there exists a constant $c>0$, independent of $n$, such that
\begin{equation}\label{eq:13.07.21_1}
\Var X_{\cT,k}(I_n)\ge cn^{d-1}.
\end{equation}
Let $\xi$ be one of the processes $\eta_{\beta}$ or $\zeta$ and let $\mu$ denotes its intensity measure. In what follows $C$ will denote a positive constant, which only depends on $d$, $k$, $\xi$, but is independent of $n$. Its exact value might be different from case to case.

We will start by introducing an additional construction. Consider the following family of sets
$$
\cK:=\{K(m)\colon m\in\ZZ^{d-1}\cap I_n\},
$$
where
$$
K(m):=\{(v,h)\in\RR^d\colon h\leq 3/32, v\in B_{b+\sqrt{3/32-h}}(m)\}
$$
for some $b\ge 1/4$. Clearly, $\#\cK=Cn^{d-1}$. In each ball $B_{1/4}(m)$, $m\in\ZZ^{d-1}\cap I_n$, we choose $d$ points $x_1(m),\ldots, x_d(m)\in\partial B_{7/32}(m)$ and such that $\Delta(m):=\conv(x_1(m),\ldots, x_d(m))$ is a regular $(d-1)$-dimensional simplex, see Figure \ref{fig:Variance} (left). Let $A(m,\xi)$ denote the event that there is exactly one point of the process $\xi$ inside each of the sets $K_i(m):=B_{\delta}(x_i(m))\times [0,1/32]\subset K(m)$, $1\leq i\leq d$, where $0<\delta<1/32$ is some fixed value, and there are no other points of the process $\xi$ inside $K(m)$. Then due to the fact that the sets $K_i(m)$, $1\leq i\leq d$, and $K(m)\setminus \bigcup_{i=1}^{d} K_i(m)$ are disjoint and using the independence properties of Poisson point process we get
\begin{align*}
\PP(A(m,\xi))&=\PP\Big(\#\big(\xi\cap(K(m)\setminus \bigcup_{1\leq i\leq d} K_i(m))\big)=0\Big)\prod_{i=1}^{d}\PP\big(\#(\xi\cap K_i(m))=1\big)\\
&=\exp\Big(-\mu\big(K(m)\setminus \bigcup_{1\leq i\leq d} K_i(m)\big)\Big)\prod_{i=1}^d\exp\big(-\mu(K_i(m))\big)\mu(K_i(m))\\
&=e^{-\mu(K(0))}\,\mu(B_{\delta}\times [0,1/32])^d,
\end{align*}
where in the last step we used additivity property of the measure $\mu$ and the invariance of $\mu$ with respect to the shift along the spacial coordinate $v$.

For $\xi=\eta_{\beta}$ we get
\begin{align*}
\mu(B_{\delta}\times [0,1/32])={c_{d,\beta}\kappa_{d-1}\delta^{d+1}\over (\beta+1)32^{1+\beta}}>0,
\end{align*}
and since $\mu$ is locally finite measure there exists a constant $C(d,\beta,b,\delta)>0$ such that
\begin{equation}\label{eq:12.07.21_2}
\PP(A(m,\eta_{\beta}))>C(d,\beta,b,\delta).
\end{equation}
For $\xi=\zeta$ we obtain
\begin{align*}
\mu(B_{\delta}\times [0,1/32])&={2\kappa_{d-1}e^{1/32}\delta^{d+1}\over (2\pi)^{d/2}}>0,
\end{align*}
and, thus, there exists a constant $C(d,b,\delta)>0$ such that
\begin{equation}\label{eq:12.07.21_3}
\PP(A(m,\zeta))>C(d,b,\delta).
\end{equation}
Finally, inside of each set $K(m)$ we place a cylinder $Z(m):=B_{\delta}(m)\times[1/32, (1/4-\delta)^2-\delta^2]$.

In the next step we denote by $\cG$ the $\sigma$-algebra that defines the positions of all points of $\xi$ except for those, which are contained in the sets $K(m)$ with ${\bf 1}(A(m,\xi))=1$. Then
\begin{equation}\label{eq:13.07.21_2}
\Var X_{\cT,k}(I_n)=\EE\Var(X_{\cT,k}(I_n)|\cG)+\Var\EE(X_{\cT,k}(I_n)|\cG)\ge \EE\Var(X_{\cT,k}(I_n)|\cG).
\end{equation}

Let us consider the conditional random variable $(X_{\cT,k}(I_n)|\cG)$ as a functional of the Poisson point process $\xi$, where in case of $\cT=\widetilde{\cD}_{\beta}$ we take $\xi=\eta_{\beta}$ and in case of $\cT=\widetilde{\cD}$ we consider $\xi=\zeta$. More formally let $\sfN(\RR^d)$ denote the space of $\sigma$-finite counting measures on $\RR^d$ and let $\xi$ be a Poisson point process on $\RR^d$ with distribution $\PP_{\xi}$ (which is a probability measure on $\sfN(\RR^d)$) and intensity measure $\mu$. We say that a random variable $F$ is a Poisson functional if $F=f(\xi)$ almost surely, where $f: \sfN(\RR^d)\mapsto\RR$ is some measurable function called a representative of $F$. Given a point $x\in\RR^d$ define the difference operator or add-one-cost operator  by
$$
D_{x}F:=f(\xi+\delta_x)-f(\xi).
$$
Applying the Fock space representation \cite[Theorem 18.6]{LP} to the Poisson functional $(X_{\cT,k}(I_n)|\cG)$ we get
\begin{align*}
\EE\Var (X_{\cT,k}(I_n)|\cG)\ge \EE\int_{\RR^d}\Big(\EE\big[D_x X_{\cT,k}(I_n)|\cG\big]\Big)^2\mu(\dint x).
\end{align*}
Since $\xi$ represents one of the processes $\eta_{\beta}$ or $\zeta$ we have that $\mu(\dint x)=f(h)\dint v\dint h$, where $f(h):\RR\mapsto [0,\infty)$ is given by \eqref{eq:intensityBeta} and \eqref{eq:intensityGaussian} respectively. Thus,
\begin{align*}
&\EE\Var (X_{\cT,k}(I_n)|\cG) \\
&\ge \EE\int_{\RR^{d-1}}\int_{\RR}\Big(\EE\big[D_{(v,h)} (X_{\cT,k}(I_n)|\cG)\big]\Big)^2f(h) \dint h\,\dint v\\
&\ge \EE\sum_{K(m)\in\cK\colon {\bf 1}(A(m))=1}\int_{B_{\delta}(m)}\int_{1/32}^{(1/4-\delta)^2-\delta^2}\Big(\EE\big[D_{(v,h)} (X_{\cT,k}(I_n)|\cG)\big]\Big)^2f(h)\dint h\,\dint v.
\end{align*}
Let us analyse the realization of the process $\xi$ inside the set $K(m)$ with ${\bf 1}(A(m))=1$. Let $(v_i,h_i):=\xi\cap K_i(m)$, $1\leq i\leq d$, be the single points that $\xi$ puts into the sets $K_i(m)$. Then $\conv(v_1,\ldots,v_d)\in \cT$, since due to construction of the sets $K(m)$, $K_i(m)$ and, conditioning on $A(m)$, the set $\Pi^{\downarrow}((v_1,h_1),\ldots, (v_d,h_d))$ is included in $K(m)$ and, hence, it cannot contain any other point of the process $\xi$. Further we note that
$$
Z(m)\subset \Big(\Pi^{\downarrow}_{(m, (1/4-\delta)^2)}\cap \big(\RR^{d-1}\times [0,\infty)\big)\Big)\subset \Pi^{\downarrow}((v_1,h_1),\ldots, (v_d,h_d)),
$$
where we recall that $Z(m)$ stands for the cylinder $B_\delta(m)\times[1/32,(1/4-\delta)^2-\delta^2]$.
The later holds, since
$$
\Pi^{\downarrow}_{(m, (1/4-\delta)^2)}\cap \RR^{d-1}=B_{1/4-\delta}(m)\subset \Pi^{\downarrow}((v_1,h_1),\ldots, (v_d,h_d))\cap \RR^{d-1},
$$
and, hence, adding a point $(v,h)\in Z(m)$ to the process $\xi$ will 'destroy' a cell $\conv(v_1,\ldots,v_d)$.

\begin{figure}[t]
    \centering
\begin{tikzpicture}
	\clip (-10.5,-2.5) rectangle (4.5,2.5);
	\node at (-6,-2.5) {\includegraphics[width=0.6\textwidth]{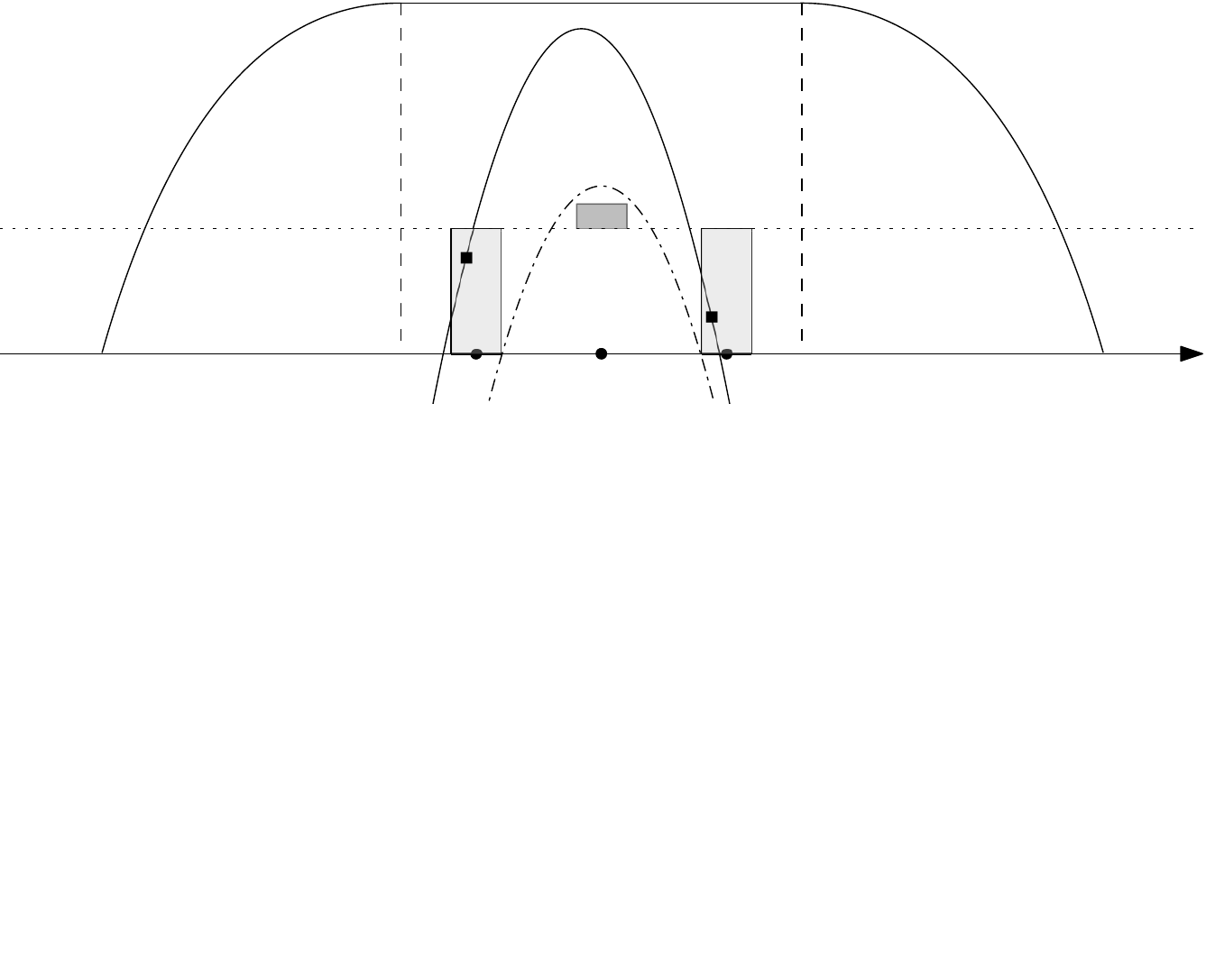}};
	\node at (2,0) {\includegraphics[width=0.3\textwidth]{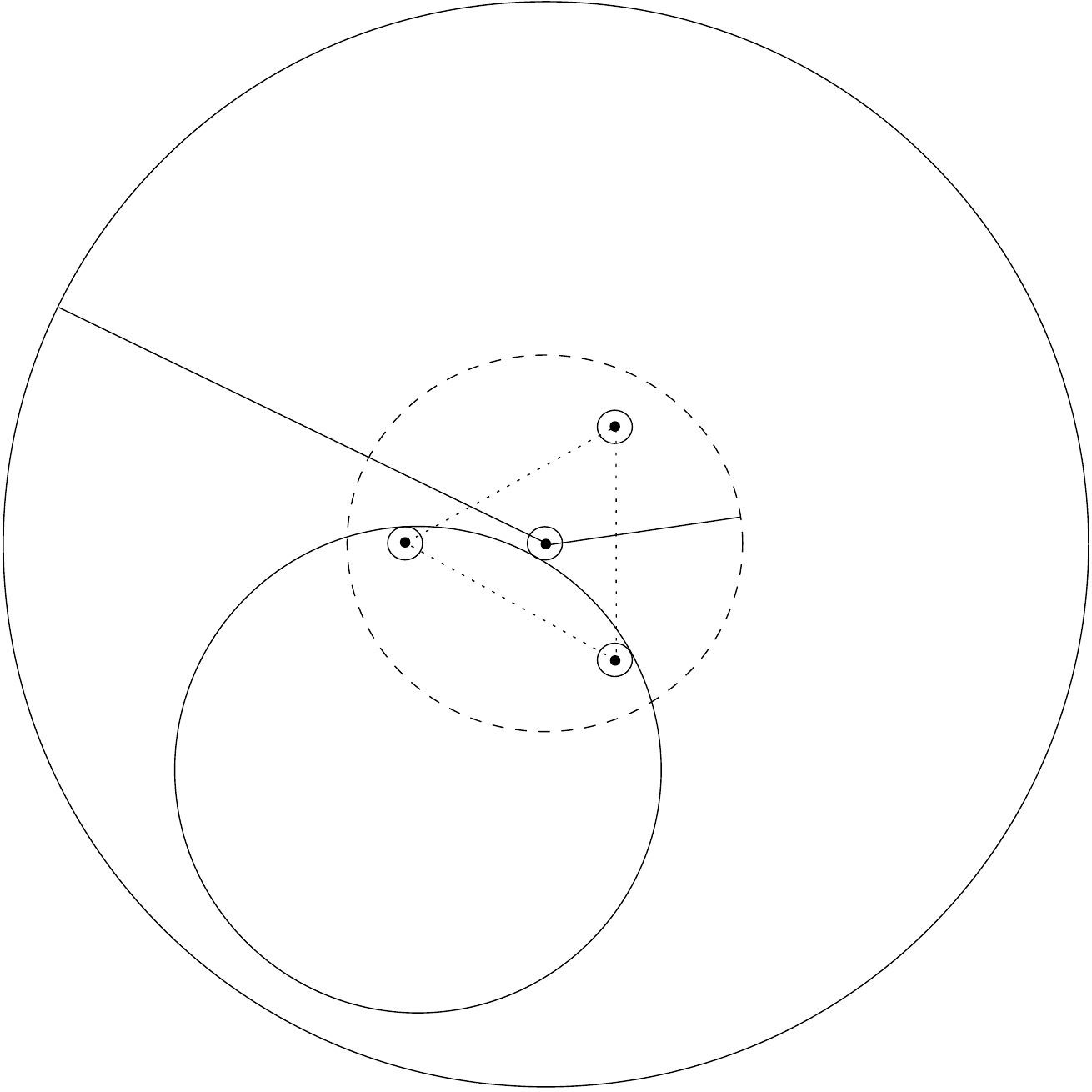}};
	\node at (-1.5,-1.7) {\tiny $\RR^{d-1}$};
	\node at (-1.5,-0.2) {\tiny $1\over 32$};
	\node at (-6,-1.6) {\tiny $m$};
	\node at (-5,1.6) {\tiny $K(m)$};
	
	\draw[->] (-5.9,-1.9) -- (-5,-1.9);
	\draw[->] (-6.1,-1.9) -- (-7,-1.9);	
	\node at (-5.5,-1.8) {\tiny $b$};
	\node at (-6.5,-1.8) {\tiny $b$};
	
	\draw[->] (-4.98,-2.2) -- (-4.8,-2.2);
	\node at (-4.9,-2.075) {\tiny $\delta$};
	\draw[->] (-5.02,-2.2) -- (-5.2,-2.2);
	\node at (-5.1,-2.075) {\tiny $\delta$};
	
	\draw[->] (-6.98,-2.2) -- (-6.8,-2.2);
	\node at (-6.9,-2.075) {\tiny $\delta$};
	\draw[->] (-7.02,-2.2) -- (-7.2,-2.2);
	\node at (-7.1,-2.075) {\tiny $\delta$};
	
	\node at (-6.1,0.4) {\tiny $Z(m)$};
	\draw (-6.2,0.3) -- (-6,-0.3);
	
	\node at (-6.8,-1.55) {\tiny $x_1(m)$};
	\node at (-4.8,-1.55) {\tiny $x_2(m)$};
	
	\node at (-8.1,-0.55) {\tiny $K_1(m)$};
	\draw (-8.1,-0.7) -- (-7,-1);
	\node at (-3.9,-0.55) {\tiny $K_2(m)$};
	\draw (-3.9,-0.7) -- (-5,-0.8);
	
	\node at (1.7,0.5) {\tiny $m$};
	\draw (1.7,0.4) -- (2,0);
	
	\node at (0.7,0.9) {\tiny $b+{1\over 4}$};
	\node at (2.6,0.2) {\tiny $b$};
	\node at (2.2,-2) {\tiny $B$};
\end{tikzpicture}
\caption{Left: Illustration of the construction within the set $K(m)$ defining the event $A(m,\xi)$ for $d=2$. The two random points in $K_i(m)$, $i\in\{1,2\}$ are indicated by the black squares. Right: Illustration of the construction within the hyperplane $\{(v,h)\in\RR^{d-1}\times\RR:h=1/32\}$ for $d=3$. Shown are the balls $B_{b+1/4}(m)$ and $B_b(m)$, the points $x_i(m)$, $i\in\{1,2,3\}$, which form a regular simplex (triangle), the balls $B_\delta(x_i(m))$ around them as well as the ball $B$.}
\label{fig:Variance}
\end{figure}
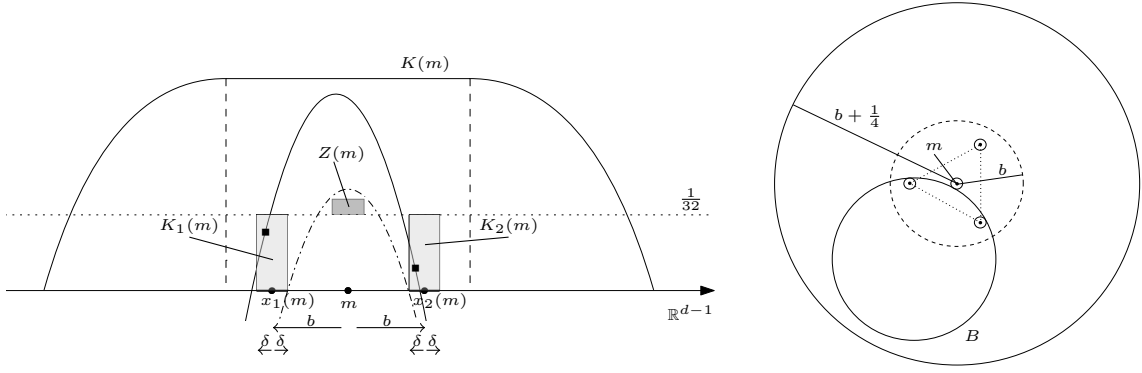
	
On the other hand for sufficiently big $b>1/4$ there exists $\delta>0$ such that for any paraboloid $\Pi\neq \Pi((v_1,h_1),\ldots, (v_d,h_d))$ under condition that $\Pi\cap \Phi(\xi)$ is a paraboloid facet of $\Phi(\xi)$ and $\Pi^{\downarrow}\cap Z(m)\neq \emptyset$ there exists $1\leq i\leq d$ such that $K_i(m)\subset \inter\Pi^{\downarrow}$. In order to prove this consider the $(d-1)$-dimensional ball $B:=\Pi^{\downarrow}\cap\{(v,h)\colon h=1/32\}$, Figure \ref{fig:Variance} (right). Then $K_i(m)\subset \inter\Pi^{\downarrow}$ if and only if $B_{\delta}(x_i(m))\subset \inter B$ and $\Pi^{\downarrow}$ intersects $Z(m)$ if and only if $B$ intersects $B_{\delta}(m)$. Moreover, since $\Pi\cap \Phi(\xi)$ is a paraboloid facet of $\Phi(\xi)$ the ball $B$ is not contained in $B_{b+1/4}(m)$ otherwise $\Pi\subset K(m)$ and, thus, $\Pi=\Pi((v_1,h_1),\ldots, (v_d,h_d))$. Taking into account that $x_i(m)$ are vertices of a regular simplex without loss of generality denote by $R$ the radius of the circumsphere of a simplex $\conv(x_1(m),\ldots,x_{d-1}(m),m)$. We note that for $b>2R-1/4$ any ball passing through $m$ and not containing any of points $x_i(m)$, $1\leq i\leq d$ in its interior is contained in $B_{b+1/4}(m)$. Since the circumradius $R$ is a continues function of $m$, $x_i(m)$, $1\leq i\leq d-1$ there exists some $\delta$-neighbourhoods of $m$ and $x_i(m)$, $1\leq i\leq d$ such that for $b=3R-1/4$ any ball intersecting $B_{\delta}(m)$ and not containing any of sets $B_{\delta}(x_i(m))$ in its interior is contained in $B_{b+1/4}(m)$, which leads to contradiction with $\Pi\neq \Pi((v_1,h_1),\ldots, (v_d,h_d))$ and the proof is completed.

The statement above means that there are no other cells in $\cT$, which could be 'destroyed' by adding a point $(v,h)\in Z(m)$ to the process $\xi$. In this case
\begin{equation}\label{eq:19-07a}
D_{(v,h)}(X_{\cT,k}(I_n)|\cG)\ge 1
\end{equation}
holds almost surely and using \eqref{eq:12.07.21_2} for $\eta_{\beta}$ and \eqref{eq:12.07.21_3} for $\zeta$ we conclude that
\begin{align*}
\Var X_{\cT,k}\geq\EE\Var (X_{\cT,k}(I_n)|\cG) &\ge \EE\sum_{K(m)\in\cK\colon {\bf 1}(A(m))=1}\mu(Z_{\varepsilon}(m))\\
&\ge C\sum_{K(m)\in\cK}\PP(A(m))\ge Cn^{d-1},
\end{align*}
which together with \eqref{eq:13.07.21_2} proves \eqref{eq:13.07.21_1}.
\end{proof}

\begin{proof}[Proof of Theorem \ref{thm:CLT}]
The proof relies on a central limit theorem for absolutely regular stationary random fields (or tessellations) in $\RR^{d-1}$ taken from \cite[Theorem 3.1]{Heinrich94} (see also \cite[Theorem 4.9]{HeinrichBook}) and which we are going to rephrase now. Suppose that for $a>0$, $b-a\ge 3$ there are functions $B_1,B_2:[0,\infty)\to[0,\infty)$ and a constant $c_0>0$ such that the absolute regularity coefficient $\mathscr{B}(a,b)$ for the stationary and random closed set $\cT$ satisfies
$$
\mathscr{B}(a,b) \leq\begin{cases}
B_1(b-a) &: b-a\geq c_0a\\
a^{d-2}B_2(b-a) &: b-a<c_0a.
\end{cases}
$$
If for some $\delta>0$
\begin{align}
\label{eq:Cond1}\EE X_{\cT,k}(I_1)^{2+\delta} & < \infty,\\
\label{eq:Cond2}\sum\limits_{r=1}^\infty r^{d-2}B_1(r)^{\delta\over\delta+2} & < \infty,\\
\label{eq:Cond3}\lim\limits_{r\to\infty}r^{2d-3}B_2(r) & =0,
\end{align}
then the limit in \eqref{eq:CLTVariance} exists in $[0,\infty]$ and \eqref{eq:CLTConvergence} holds. So, it remains to check \eqref{eq:Cond1}, \eqref{eq:Cond2} and \eqref{eq:Cond3} for the two tessellations $\widetilde{\cD}_\beta$ and $\widetilde{\cD}$.
\begin{itemize}
\item[1.] If $\cT=\widetilde{\cD}_\beta$ for some $\beta>-1$, we may take an arbitrary $\delta\in\NN$. Then \eqref{eq:Cond1} is satisfied since $X_{\widetilde{\cD}_\beta,k}(I_1)$ has finite moments of all orders thanks to Lemma \ref{lem:MomentsForCLT}. According to Theorem \ref{thm:BetaEstimates} the functions $B_1$ and $B_2$ can be taken as
\begin{align*}
B_1(r) &:= c_7\,r\,e^{-c_6\,r^{d+1+2\beta}},\\
B_2(r) &:= \tilde{c}_7\, r^{-(d-3)}\,e^{-\tilde{c}_7\, r^{d+1+2\beta}},
\end{align*}
where the constants are the same as in Theorem \ref{thm:BetaEstimates}. Then \eqref{eq:Cond2} and \eqref{eq:Cond3} are trivially satisfied.

\item[2.] If $\cT=\widetilde{\cD}$ we may take again an arbitrary $\delta\in\NN$. Then \eqref{eq:Cond1} is satisfied as $X_{\widetilde{\cD},k}(I_1)$ has finite moments of all orders by Lemma \ref{lem:MomentsForCLT} again. Moreover, In view of Theorem \ref{thm:BetaEstimates} the functions $B_1$ and $B_2$ can be taken as
\begin{align*}
B_1(r) &= c_{10}\,e^{-c_9r^2},\\
B_2(r) &= \tilde{c}_{10}\,r^{-(d-2)}\,e^{-c_9r^2},
\end{align*}
where the constants are again the same as in Theorem \ref{thm:BetaEstimates}. In this case conditions \eqref{eq:Cond2} and \eqref{eq:Cond3} are again trivially satisfied.
\end{itemize}
This completes the proof of \eqref{eq:CLTConvergence}. The strict positivity of $s_k^2$ has been shown in Lemma \ref{lem:sk>0}. The proof is thus complete.
\end{proof}

\begin{remark}\label{rem:CLTbetaPrime}\rm
	One might wonder whether the argument in the proof of Theorem \ref{thm:CLT} can also be used to prove a central limit theorem for the number of $k$-dimensional faces of $\beta'$-Delaunay tessellations. First, by Theorem \ref{thm:BetaEstimates} the two functions $B_1$ and $B_2$ are
	\begin{align*}
		B_1(r) &= c_8\,r^{-(2\beta-d-1)}\\
		B_2(r) &= \tilde{c}_8\,r^{-2\beta+3}
	\end{align*}
	with the same constants as in Theorem \ref{thm:BetaEstimates}. Then \eqref{eq:Cond3} is satisfied for $\beta>d$ and \eqref{eq:Cond2} is equivalent to $d+\delta d<1+\delta\beta$. For example, taking $\delta=1$ this inequality holds once $\beta>d$ is satisfied. However, using the argument from the proof of Theorem \ref{thm:CLT} we were not able to verify \eqref{eq:Cond1} for $k=0$ and with $\delta=1$, for example. In contrast to the $\beta$- and the Gaussian Delaunay tessellation, the concentration bound in Lemma \ref{lm:4_3} for the $\beta'$-Delaunay tessellation is not strong enough to allow such a conclusion, as the resulting integral diverges. In addition, the transition from $k=0$ to $k\geq 1$ would require $\sup_{v\in B_\varrho}\EE f_k(v)^{2m}$ to be finite at least for $m=3$, recall \eqref{eq:19-01-21a}. Again in contrast to the $\beta$- and the Gaussian Delaunay tessellation, to prove such a result for the $\beta'$-Delaunay tessellation seems a non-trivial task and presumably requires further assumptions on the model parameter $\beta$. A similar comment applies to Theorem \ref{thm:CLTkvolume} below as well.
\end{remark}

As a second application we discuss the central limit theorem for the $k$-volume of the $k$-skeleton of $\beta$- and Gaussian-Delaunay tessellations. To introduce the set-up fix $k\in\{0,1,\ldots,d-2\}$ and for a stationary and isotropic random tessellation $\cT$ in $\RR^{d-1}$ let $\cT_k$ be the union of all $k$-dimensional faces of all cells of $\cT$. We refer to $\cT_k$ as the $k$-skeleton of $\cT$. The $k$-dimensional Hausdorff measure $\cH^k$ restricted to $\cT_k$ is a stationary and isotropic random measure in $\RR^{d-1}$ and we denote by $\nu_{\cT,k}$ its intensity. In other words, $\nu_{\cT,k}$ is the $k$-volume ($k$-dimensional Hausdorff measure) of the $k$-skeleton $\cT_k$ per unit $(d-1)$-volume. We are interested in the random variables $Y_{\cT,k}(I_n):=\cH^k(\cT_k\cap I_n)$, $n\in\NN$.

\begin{theorem}\label{thm:CLTkvolume}
	Let $\cT$ be one of the tessellations $\widetilde{\cD}_\beta=\cD(\eta_\beta)$ for $\beta>-1$ or $\widetilde{\cD}=\cD(\zeta)$. Further, fix $k\in\{0,1,\ldots,d-1\}$. Then the limit
	\begin{equation}\label{eq:CLTVarianceVol}
		t_k^2 := \lim_{n\to\infty}(2n)^{-(d-1)}\Var Y_{\cT,k}(I_n)
	\end{equation}
	exists and is strictly positive. Moreover, we have the convergence in distribution
	\begin{equation}\label{eq:CLTConvergenceVol}
		{Y_{\cT,k}(I_n)-(2n)^{d-1}\nu_{\cT,k}\over (2n)^{d-1\over 2}}\overset{d}{\longrightarrow} G,
	\end{equation}
	as $n\to\infty$, where $G\sim\mathcal{N}(0,t_k^2)$ is a centred Gaussian random variable with variance $t_k^2$.
\end{theorem}
\begin{proof}
We may represent the random variables $(Y_{\cT,k}(I_n)-(2n)^{d-1}\nu_{\cT,k})/(2n)^{d-1\over 2}$ as $(2n)^{-{d-1\over 2}}S_n$, where
$$
S_n:=\sum_{m\in\ZZ^{d-1}\cap I_n}[\cH^k(\cT_k\cap([0,1]^{d-1}+m))-\nu_{\cT,k}],
$$
and are thus in the position to apply the central limit theorem of Heinrich for absolutely regular stationary random fields \cite[Theorem 6.1]{Heinrich94} or \cite[Theorem 4.9]{HeinrichBook}. To this end, conditions \eqref{eq:Cond1}--\eqref{eq:Cond3} need to be satisfied for $X_{\cT,k}$ replaced by $Y_{\cT,k}$. While \eqref{eq:Cond2} and \eqref{eq:Cond2} have already been checked in the proof of Theorem \ref{thm:CLT}, it remains to deal with \eqref{eq:Cond1}. Since $\cT_k$ is almost surely the union of $k$-dimensional simplices, for any $m\in\NN$ it follows that
\begin{align*}
\EE Y_{\cT,k}(I_1)^m \leq \EE\Big(\max_E\cH^k(E\cap I_1)X_{\cT,0}(I_1)\sup_{v\in I_1}f_k(v)\Big)^m,
\end{align*}
where the maximum is taken over all $E$-dimensional affine subspaces in $\RR^{d-1}$, $X_{\cT,0}(I_1)$ is the number of vertices of $\cT$ in $I_1$ and $f_k(v)$ denotes the number of $k$-faces of $\cT$ adjacent to $v$. Putting $\varrho:=\sqrt{d-1}$, we have $\max_E\cH^k(E\cap I_1)\leq\kappa_k(d-1)^{k\over 2}$ and hence
$$
\EE Y_{\cT,k}(I_1)^m \leq \kappa_k^m(d-1)^{km\over 2}\EE\Big(X_{\cT,0}(B_\varrho)\sup_{v\in B_\varrho}f_k(v)\Big)^m.
$$
The last expression is finite for both choices $\cT=\widetilde{\cD}_\beta$ and $\cT=\widetilde{\cD}$ according to the discussion after \eqref{eq:19-01-21a}.

The strict positivity of the limiting variance $t_k^2$ follows along the lines of the proof of Lemma \ref{lem:sk>0}. In fact, the only change occurs at \eqref{eq:19-07a}, which has to be replaced by
$$
D_{(v,h)}(X_{\cT,k}(I_n)|\cG)\ge c_{d,k,\delta},
$$
where the constant $c_{d,k,\delta}>0$, which only depends on $d$, $k$ and $\delta$, is the smallest possible $k$-volume of a $k$-face of a simplex with vertices in $B_\delta(x_i(m))$, where the points $x_i(m)$, $1\leq i\leq d$, are the vertices of a regular simplex with vertices on a sphere of radius $7/32$ (compare with the construction at the beginning of Lemma \ref{lem:sk>0}. This completes the argument.
\end{proof}

\subsection*{Acknowledgement}
Most of this work was carried out during the Trimester Program \textit{The Interplay between High Dimensional Geometry and Probability} at the Hausdorff Research Institute for Mathematics in Bonn.\\
ZK was supported by the DFG under Germany's Excellence Strategy  EXC 2044 -- 390685587, \textit{Mathematics M\"unster: Dynamics - Geometry - Structure}. CT and ZK were supported by the DFG via the priority program \textit{Random Geometric Systems}.


\begin{thebibliography}{30}\small

\bibitem{Bradley}
R.C.~Bradley,
Basic properties of strong mixing conditions. A survey and some open questions,
Probab. Surveys \textbf{2}, 107--144 (2005).

\bibitem{CalkaSchreiberYukich}
P.~Calka, T.~Schreiber and J.E.~Yukich,
Brownian limits, local limits and variance asymptotics for
convex hulls in the ball,
Ann. Probab. \textbf{41},  50--108 (2013).

\bibitem{CalkaYukichGaussian}
P.~Calka and J.E.~Yukich,
Variance asymptotics and scaling limits for {G}aussian
polytopes,
Probab. Theory Related Fields \textbf{163},  259--301 (2015).

\bibitem{DVJ}
D.J.~Daley and D.~Vere-Jones,
\emph{An Introduction to the Theory of Point Processes, Vol. II, General Theory and Structure},
2nd edition (2008), Springer-Verlag.

\bibitem{GKT1}
A.~Gusakova, Z.~Kabluchko and C.~Th\"ale,
The $\beta$-Delaunay tessellation I: Description of the model and geometry of typical cells,
arXiv: 2005.13875.

\bibitem{GKT2}
A.~Gusakova, Z.~Kabluchko and C.~Th\"ale,
The $\beta$-Delaunay tessellation II: Convergence to the Gaussian limit tessellation,
arXiv: 2101.11316.

\bibitem{GKT3}
A.~Gusakova, Z.~Kabluchko and C.~Th\"ale,
The $\beta$-Delaunay and Gaussian-Delaunay tessellation III: Kendall's problem and limit theorems in high dimensions,
arXiv: 2104.07348.


\bibitem{Heinrich94}
L.~Heinrich,
Normal approximation for some mean-value estimates of absolutely regular tessellations,
Math. Methods Statist. \textbf{3}, 1--24 (1994).

\bibitem{HeinrichBook}
L.~Heinrich, Asymptotic methods in statistics of random point processes,
in \textit{Stochastic Geometry, Spatial Statistics and Random Fields}, Lecture Notes in Math. \textbf{2068} (2013), Springer.

\bibitem{KMK}
J.~Kerstan, K.~Matthes and J.~Mecke,
\emph{Infinitely Divisible Point Processes},
Wiley Series in Probability \& Mathematical Statistics (1978), John Wiley \& Sons Ltd.

\bibitem{LP}
G. Last and M. Penrose,
\emph{Lectures on the {P}oisson {P}rocess}, vol.~7 of Institute
  of Mathematical Statistics Textbooks,
Cambridge University Press, Cambridge, 2018.

\bibitem{MN15}
S.~Martinez and W.~Nagel,
The $\beta$-mixing rate of STIT tessellations,
Stochastics: An International Journal of Probability and Stochastic Processes, \textbf{88}, 396--414 (2015).

\bibitem{Matheron}
G.~Matheron,
\emph{Random Sets and Integral Geometry},
Wiley Series in Probability \& Mathematical Statistics (1975), John Wiley \& Sons Ltd.

\bibitem{NIST}
F.W.J.~Olver, D.W.~Lozier, R.F.~Boisvert, and C.W.~Clark,
\emph{NIST {H}andbook of {M}athematical {F}unctions},
Cambridge University Press, 2010.

\bibitem{ReitznerCLT}
M. Reitzner, Central limit theorems for random polytopes, Probab. Theory Related Fields \textbf{133},  483--507 (2005).

\bibitem{SW}
R.~Schneider and W.~Weil,
\emph{Stochastic and Integral Geometry},
Probability and its Applications (2008), Springer-Verlag.


\bibitem{SchreiberSurvey}
T.~Schreiber, Limit theorems in stochastic geometry, in \textit{New Perspectives in Stochastic Geometry}, 111--144, Oxford University Press (2010).

\bibitem{SchreiberYukich}
T.~Schreiber and J.E.~Yukich,
Variance asymptotics and central limit theorems for generalized growth processes with applications to convex hulls and maximal points,
Ann. Probab. \textbf{36}, 363--396 (2008).

\bibitem{VershyninBook}
R.~Vershynin,
\emph{High-Dimensional Probability},
Cambridge University Press, 2018.

\bibitem{YukichSurvey}
J.E.~Yukich, Limit theorems in discrete stochastic geometry, in \textit{Stochastic Geometry, Spatial Statistics and Random Fields}, Lecture Notes in Math. \textbf{2068}, 239--275 (2013).

\end{thebibliography}
\end{document}